\documentclass[a4paper,fleqn]{cas-sc}
\usepackage[utf8]{inputenc}

\usepackage[numbers]{natbib}
\usepackage{titlesec}
\titleformat{\section}[block]
  {\normalfont\bfseries}
  {\thesection}{1em}{}
\titleformat{\subsection}[block]
  {\normalfont\itshape}
  {\thesubsection}{1em}{}
\titleformat{\subsubsection}[block]
  {\normalfont\itshape}
  {\thesubsubsection}{1em}{}

\usepackage{type1cm}
\usepackage{anyfontsize}

\DeclareMathAlphabet{\mathboondoxcal}{U}{BOONDOX-cal}{m}{n}
\newcommand{\jay}{\mathboondoxcal{j}}

\usepackage{amsthm}
\usepackage{mathtools}
\usepackage{colortbl}
\usepackage{stmaryrd}
\usepackage{scalefnt}

\newtheorem{theorem}{Theorem}
\newtheorem{corollary}[theorem]{Corollary}
\newtheorem{assumption}{Assumption}
\newtheorem{lemma}[theorem]{Lemma}

\theoremstyle{remark}
\newtheorem{rmk}{Remark}

\theoremstyle{definition}
\newtheorem{problem}{Problem}

\numberwithin{equation}{section}
\numberwithin{theorem}{section}
\numberwithin{problem}{section}

\usepackage[compatibility=false]{caption}
\usepackage{subcaption}
\usepackage{threeparttable}

\usepackage{tikz}
\usepackage{pgfplots}
\usetikzlibrary{calc, arrows.meta, positioning}
\usepackage{booktabs}
\usepackage{siunitx}
\usepackage{makecell}
\usepackage{pgfplotstable}
\usepackage{booktabs}
\usepgfplotslibrary{groupplots}
\pgfplotsset{colormap={vibrant}{rgb255=(68,1,84) rgb255=(33,144,141) rgb255=(253,231,37)}}
\pgfplotsset{compat=1.18}
\usepackage{float}
\usepackage{makecell}
\pgfplotsset{
  colormap={viridis_exact}{
    rgb=(0.267004, 0.004874, 0.329415)  
    rgb=(0.282623, 0.140926, 0.457517)  
    rgb=(0.253935, 0.265254, 0.529983)  
    rgb=(0.207063, 0.371774, 0.553117)  
    rgb=(0.163625, 0.471133, 0.558148)  
    rgb=(0.127568, 0.566949, 0.550556)  
    rgb=(0.134692, 0.658636, 0.517649)  
    rgb=(0.208030, 0.748751, 0.472873)  
    rgb=(0.327796, 0.827199, 0.406859)  
    rgb=(0.477504, 0.821444, 0.318195)  
    rgb=(0.627976, 0.855711, 0.223729)  
    rgb=(0.741388, 0.873449, 0.149561)  
    rgb=(0.845561, 0.887322, 0.099702)  
    rgb=(0.926106, 0.896785, 0.107893)  
    rgb=(0.993248, 0.906157, 0.143936)  
  }
}
\definecolor{sandstone}{RGB}{194,192,63}
\definecolor{siltloam}{RGB}{106,148,63}
\definecolor{beitnetofaclay}{RGB}{105,104,180}

\definecolor{ddH1}{HTML}{08306b}
\definecolor{ddH2}{HTML}{2171b5}
\definecolor{ddH3}{HTML}{6baed6}
\definecolor{ddH4}{HTML}{bdd7e7}
\definecolor{lsH1}{HTML}{67000d}
\definecolor{lsH2}{HTML}{cb181d}
\definecolor{lsH3}{HTML}{ef3b2c}
\definecolor{lsH4}{HTML}{fc9272}
\definecolor{picH1}{HTML}{00441b}
\definecolor{picH2}{HTML}{238b45}
\definecolor{picH3}{HTML}{74c476}
\definecolor{picH4}{HTML}{c7e9c0}
\definecolor{newH1}{HTML}{7f2704}
\definecolor{newH2}{HTML}{d94801}
\definecolor{newH3}{HTML}{f16913}
\definecolor{newH4}{HTML}{fdae6b}
\usepackage{enumitem}

\usepackage{hyperref}
\usepackage{cleveref}
\crefname{notation}{notation}{notations}
\Crefname{notation}{Notation}{Notations}
\Crefname{problem}{Problem}{Problems}
\crefname{lemma}{lemma}{lemmas}
\Crefname{lemma}{Lemma}{Lemmas}

\def\tsc#1{\csdef{#1}{\textsc{\lowercase{#1}}\xspace}}
\tsc{WGM}
\tsc{QE}
\tsc{EP}
\tsc{PMS}
\tsc{BEC}
\tsc{DE}

\begin{document}
\input{figures.tex}

\let\WriteBookmarks\relax
\def\floatpagepagefraction{1}
\def\textpagefraction{.001}
\shorttitle{mLRDD-scheme for Richards' Equation}
\shortauthors{Synnev\aa g et~al.}

\title[mode=title]{Linear Robin-type Domain Decomposition scheme for Mixed Formulation of Richards' Equation}
\author[1]{\AA smund v.B. Synnev\aa g}[orcid=0009-0002-2407-6539]
\cormark[1]
\ead{asmund.synnevag@uib.no}

\author[2]{Wietse M. Boon}[orcid=0000-0003-4080-2369]

\author[1]{Florin A. Radu}[orcid=0000-0002-2577-5684]

\author[3]{Sarah E. Gasda}[orcid=0000-0002-2610-8322]

\affiliation[1]{organization={Department of Mathematics, University of Bergen},
                addressline={Allégaten 41}, 
                postcode={5007}, 
                postcodesep={}, 
                city={Bergen},
                state={Vestland},
                country={Norway}}

\affiliation[2]{organization={Faculty of Mathematics, University of Duisberg-Essen},
                addressline={Thea-Leymann-Straße 9}, 
                postcode={45127}, 
                postcodesep={}, 
                city={Essen},
                country={Germany}}

\affiliation[3]{organization={NORCE Norwegian Research AS},
                addressline={Nygårdsgaten 112}, 
                postcode={5008}, 
                postcodesep={}, 
                city={Bergen},
                state={Vestland},
                country={Norway}}

\cortext[cor1]{Corresponding author:}

\begin{abstract}
    In this work, we present a linear Robin-type domain decomposition scheme for solving the mixed formulation of Richards' equation (mLRDD-scheme), governing flow in variably saturated porous media. Assuming a highly heterogeneous porous medium consisting of multiple blocks/layers of distinct materials, we apply non-overlapping domain decomposition to isolate individual materials, yielding near-homogeneous conditions within each subdomain. Richards' equation is discretized in time by backward Euler and linearized using the L-scheme. A Robin-type interface condition, consistent with the mixed formulation, is then derived to couple the subdomains. For spatial discretization, mixed finite elements are employed, ensuring local mass conservation and providing the flux variables consistent with the Robin condition. The convergence of the proposed scheme is rigorously proved. Heterogeneous, multidomain numerical examples in 2D/3D are presented to demonstrate the efficiency and robustness of the scheme, along with numerical validation of the theoretical convergence results.
\end{abstract}


\begin{NoHyper}
\maketitle
\end{NoHyper}

\section{Introduction} \label{sec:introduction}
    Flow in porous media has many relevant applications, such as soil pollution tracking, CO$_2$ sequestration, nuclear waste management, enhanced geothermal energy extraction, and cancer research, among others. When modeling flow in a variably saturated porous medium, where the pore space is shared by water and air, assuming constant pressure for the air phase yields Richards' equation \cite{richards1931, richardson1922weather}. This formulation captures complex physical behaviors, such as sharp wetting fronts and steep pressure gradients, causing the governing equation to dynamically shift between elliptic in fully saturated zones and parabolic in unsaturated zones. Consequently, Richards' equation is a highly nonlinear, degenerate parabolic-elliptic equation, and the creation of efficient, robust numerical methods remains a major challenge \cite{farthing2017numerical}. Meeting this challenge in practice requires taking three things into consideration: a spatial discretization that stays physically faithful to the underlying flow and mass, a linearization scheme robust enough for the resulting nonlinear systems, and, in realistic heterogeneous geology, a strategy for handling highly contrasting material properties. Each of these aspects has been well studied in isolation. This paper, however, explores a solution that handles all three aspects simultaneously. 

    There are many works concerning discretization techniques for Richards' equation. For temporal discretization, backward Euler is used almost exclusively, as higher-order schemes are not recommended for solutions of such low regularity as for Richards' equation. For spatial discretization, Galerkin finite elements \cite{aavatsmark1998discretization, arbogast1993numerical, nochetto1988approximation, slodicka2002robust}, finite volumes \cite{manzini2004mass, eymard1999finite, eymard2006combined}, or mixed finite element methods (MFEM) \cite{arbogast1996nonlinear, radu2004order, bergamaschi1999mixed, woodward2000analysis} are commonly used. MFEM are particularly attractive due to their direct evaluation of the mass balance equation, giving mass conservation with appropriate function spaces.  Furthermore, MFEM treats both pressure and flux as primary variables, giving a direct computation of the flux as a primary unknown, rather than via post-processing steps. These attributes come at a price, however: expressing Richards' equation in mixed form couples pressure and flux into a nonlinear saddle-point system at every time step, unlike the single-field systems produced by a primal, pressure-only formulation.
    
    Solving the nonlinear saddle-point system at each timestep requires the choice of a linearization scheme, where the scheme is just as decisive for the robustness and efficiency of the simulation as the spatial discretization itself. For Richards' equation, two main linearization schemes are typically used: the quadratically but only locally convergent Newton method (this includes the nested variant of \cite{casulli2010}), and linearly convergent but more robust fixed-point schemes; this includes the Picard method, the modified Picard method \cite{celia1990general}, and the L-scheme \cite{list2016study,pop2004mixed} (including its modified variant \cite{mitra2019modified}). While the Newton method offers quadratic convergence, its application in the mixed setting requires the construction of complex derivative matrices \cite{bergamaschi1999mixed}, which can lead to ill-conditioned systems. The L-scheme elegantly bypasses this by replacing the constitutive derivatives with a fixed stabilization constant, making it unconditionally convergent under mild constraints \cite{stokke2023adaptive}, even for an inaccurate initial guess. However, as recently shown by Maier et al. \cite{maier2023}, no single linearization method -- neither the L-scheme, Newton, nor Picard -- is uniformly robust across realistic soil types and boundary conditions. In other words, the nonlinearities of the underlying problem alone do not cause a scheme to fail, but also how it interacts with the heterogeneity of the domain in question. This highlights a critical challenge: how to maintain MFEM's faithfulness to the underlying physics while achieving robust convergence over heterogeneous domains. 
    
    The introduction of heterogeneity to a system does not only cause a minor complication. Real-life subsurface domains are inherently heterogeneous, typically layered or composed of blocks of different soils, e.g., sand, clay, or loam, resulting in sharp material interfaces. These sharp interfaces can exacerbate the governing nonlinearities, further limiting the use of lower-order methods. One solution for lower-order methods is enhancement via acceleration (e.g., Anderson acceleration \cite{Anderson1965, stokke2023adaptive}) or hybridization with Newton \cite{bergamaschi1999mixed, list2016study} to ensure reliable convergence in such complex regimes. Another approach is to apply domain decomposition (DD) methods, where domains are split into an arbitrary number of subdomains. By splitting along sharp material interfaces, DD methods handle the heterogeneity naturally,  limiting the material interaction within a subdomain, isolating severe nonlinearities within local subproblems, and coupling the subproblems through interface conditions. One of the fundamental DD approaches is through the classical overlapping Schwarz methods \cite{lions1988schwarz, gander2005overlapping}. Overlapping Schwarz has been shown by \cite{lions1988schwarz} to converge for nonlinear elliptic problems, though the method exhibits slow convergence. Also, the overlapping of different material regions defeats the purpose of solving local, close-to-homogeneous problems. Another, more suitable approach is the non-overlapping methods. There has been research on Dirichlet-Neumann, Neumann-Neumann variants for parabolic problems \cite{gander2013dirichlet}, optimal transmission conditions \cite{gander2015optimized, gander2003optimal}, and Robin-type variants \cite{gander2015optimized, lions1990schwarz}. The Robin-type methods provide faster and more robust alternatives, at the cost of careful parameter tuning and interface coupling. Here, not only the parameters, but the parameterizations (i.e., constitutive relations), differ across the interface. This demands a careful formulation of the interface conditions \cite{ahmed2019posteriori, seus2018linear}, such that the subproblems adhere to the underlying physics and functional setting. 
    
    For the mixed formulation of Richards' equation, the Robin-type interface condition requires well-defined normal fluxes across subdomain boundaries — both to respect the continuity of the physical flux and to satisfy the pressure-flux coupling in the appropriate functional framework. This is precisely what MFEM provides: by placing the flux and pressure in suitable regular spaces, accurate normal fluxes arise as part of the primary unknowns, continuous across element boundaries and without post-processing. This gives a natural pairing between MFEM and non-overlapping Robin-type domain decomposition. Moreover, MFEM provides well-established \textit{a priori} error estimates, giving us a theoretical foundation for validation. At the semi-discrete level, Seus et al. \cite{seus2018linear} demonstrated that Robin-type interface conditions can be successfully combined with the L-scheme for the primal formulation of Richards' equation, where pressure is the primary variable. They proved convergence of the scheme in a semi-discrete setting and performed numerical simulations using a cell-centered two-point flux approximation variant of the finite volume method. We extend this work to the mixed formulation, and with the following steps, we derive the proposed mixed L-scheme Robin-type domain decomposition scheme (mLRDD-scheme). By applying non-overlapping Robin-type domain decomposition to the mixed formulation of Richards' equation, we localize and decouple the nonlinearities. By applying the L-scheme, we bypass the need for exact Jacobians, ensuring robust convergence even across sharp material interfaces. By discretizing using MFEM, with $H(\mathrm{div};\Omega)-L^2(\Omega)$ conforming elements, we achieve local mass conservation and naturally obtain accurate quantities for the Robin-type interface conditions, while giving a solid theoretical foundation.
    
    To summarize, the new contributions of this paper are:
    
    \begin{itemize}
    \item{a new non-overlapping domain decomposition MFEM/L-scheme scheme, with flux-based Robin-type interface conditions, for mixed formulation of Richards' equation in heterogeneous domains}
    \item{a rigorous convergence analysis of the new scheme}
    \item{a numerical comparison of the new scheme in relation to a monolithic approach with Newton, Picard, and L-scheme type linearizations for multiple heterogeneous domains, in two and three spatial dimensions}
    \end{itemize}
    
    The paper is structured as follows: in Section~\ref{sec:problem} we introduce the notation, functional setting, the equations, temporal discretization and linearization, concluding with the iterative mLRDD-scheme; in Section~\ref{sec:analysis} we analyse the convergence of the proposed scheme; Section~\ref{sec:numeric} presents two- and three-dimensional numerical examples, giving a comprehensive comparison to monolithic schemes; Section~\ref{sec:conclusion} ends the paper with a concluding section.

\section{Problem formulation and notation conventions}\label{sec:problem}

    In this paper, we will use common notations from functional analysis. For a complete summary of the functional setting utilized throughout this paper, see Appendix~\ref{sec:appendix}. The following section presents the mixed formulation of Richards' equation in the continuous setting. Later, we will apply non-overlapping domain decomposition, followed by temporal discretization and linearization. The weak formulation is then introduced, together with the derivation of the Robin-type interface conditions. Lastly, the final iterative scheme, the mLRDD scheme, is presented.

\subsection{The mixed formulation of Richards' equation}
     Let $\Omega \subset \mathbb{R}^d $, $d\in\{ 2,3 \}$ be an open, bounded domain with a Lipschitz-continuous boundary $\partial\Omega$. We assume that the boundary is partitioned into a Dirichlet part $\partial\Omega_D$ and a Neumann part $\partial\Omega_N$ such that $\partial\Omega = \overline{\partial\Omega}_D \cup \overline{\partial\Omega}_N$ and $\partial\Omega_D \cap \partial\Omega_N = \emptyset$. Richards' equation, in its mixed pressure-flux formulation, reads
    \begin{subequations} \label{eq:1}
        \begin{alignat}{2}
            \hat{\mathbf{q}} &= -\frac{\mathbf{K}}{\mu} k_r\bigl(S(p)\bigr)\nabla\!\bigl(p - \rho g_z z \bigr) &\qquad\text{in } \Omega\times(0,T], \label{eq:1a}\\
            \phi \partial_t S(p) + \nabla \cdot \hat{\mathbf{q}} &= f &\qquad\text{in } \Omega\times(0,T], \label{eq:1b}
        \end{alignat}
    \end{subequations}
    where \eqref{eq:1a} represents the Darcy flow equation and \eqref{eq:1b} represents the mass balance equation, with water pressure $p$ and flux $\hat{\mathbf{q}}$ as primary unknowns. $T$ denotes the final time. The coefficient functions $k_r(\cdot)$ and $S(\cdot)$ are the nonlinear, pressure-dependent relative permeability and water saturation, respectively. The known physical parameters are porosity $\phi$, intrinsic permeability $\mathbf{K}$, water density and viscosity $\rho$ and $\mu$, gravitational acceleration $g_z$, and the source term $f$. Table (\ref{tab:Notations}) gives a comprehensive list of notation and can be seen in Appendix \ref{sec:appendix}. Assuming constant porosity, we define the collective terms
    \begin{align}
        \mathbf{q} &\coloneqq \frac{\hat{\mathbf{q}}}{\phi}, &
        k^{-1}(S(p)) &\coloneqq \left[\frac{\mathbf{K}}{\mu \phi} k_r\bigl(S(p)\bigr)\right]^{-1}, &
        \eta &\coloneqq \rho g_z z, & 
        \mathcal{F} &\coloneqq \frac{f}{\phi}.
    \end{align}
    
    With the addition of homogeneous boundary and initial conditions, using the defined terms in \eqref{eq:1}, the strong formulation of Richards' equation becomes
    \begin{subequations} \label{eq:2}
        \begin{align}
            k^{-1}(S(p))\ \mathbf{q} + \nabla p &= \nabla \eta & \text{in } &\Omega\times(0,T]\label{eq:2a}, \\
            \partial_t S(p) + \nabla \cdot \mathbf{q} &= \mathcal{F} & \text{in } &\Omega\times(0,T] \label{eq:2b}, \\
            \mathbf{q} \cdot \mathbf{n} &= 0 &\text{on } &\partial \Omega_N \times(0,T], \\
            p &= 0 &\text{on } &\partial \Omega_D \times(0,T], \\
            p &= p_0 &\text{in } &\Omega \times \{0\},
        \end{align}
    \end{subequations}
    where $p_o \in L^2(\Omega)$ is a given function. In the remainder of \Cref{sec:problem} and the following \Cref{sec:analysis}, the boundary and initial conditions in \eqref{eq:2} are inherited, although they have been omitted for brevity. For simplicity, we restrict the above formulation to homogeneous boundary conditions. The extension to non-homogeneous boundary conditions proceeds analogously. 
    
\subsection{Non-overlapping domain decomposition}
     The domain $\Omega$ is partitioned into two non-overlapping Lipschitz subdomains $\Omega_1$ and $\Omega_2$ such that $\overline{\Omega} = \overline{\Omega}_1 \cup \overline{\Omega}_2$ and $\Omega_1 \cap \Omega_2 = \emptyset$. 
    Let $\partial\Omega_1$ and $\partial \Omega_2$ denote their respective boundaries, separated by the internal interface $\Gamma \coloneqq \partial \Omega_1 \cap \partial \Omega_2$, which forms a $(d-1)$-dimensional manifold in $\overline{\Omega}$. For simplicity, we assume that each subdomain borders (part of) the Dirichlet boundary, i.e., $\partial\Omega_\ell \cap \partial \Omega_D \ne \emptyset$ for each subdomain index $\ell \in \{1,2\}$. 
    Let $\mathbf{n}_\ell$ denote the outward unit normal vector on $\partial\Omega_\ell$; in particular, $\mathbf{n}_1 = -\mathbf{n}_2$ on $\Gamma$. Richards' equation on subdomain $\Omega_\ell$ is then given by 
    \begin{subequations} \label{eq:3}
        \begin{alignat}{2} 
            k_\ell^{-1}(S_\ell(p_\ell)) \, \mathbf{q}_\ell + \nabla p_\ell &= \nabla \eta_\ell &\qquad\text{in } \Omega_\ell\times(0,T], \\ 
            \partial_t S_\ell(p_\ell) + \nabla \cdot \mathbf{q}_\ell &= \mathcal{F}_\ell &\qquad\text{in } \Omega_\ell\times(0,T], \\
            \mathbf{q}_1 \cdot \mathbf{n}_1 &= -\mathbf{q}_2 \cdot \mathbf{n}_2 &\qquad\text{on } \Gamma\times(0,T], \label{eq:3c}\\
            p_1 &= p_2 &\qquad\text{on } \Gamma\times(0,T], \label{eq:3d}
        \end{alignat}
    \end{subequations}
    with $\ell \in \{1,2\}$ being the subdomain index. Here, and throughout this work, a subscript $\ell$ on a variable denotes its restriction to $\Omega_\ell$. The transmission conditions \eqref{eq:3c} and \eqref{eq:3d} ensure the continuity of flux and pressure across the interface $\Gamma$. We impose the following assumptions on $k_\ell$ and $S_\ell$.

    \begin{assumption}{} \label{AssumptionOne}
        For each subdomain $\ell\in\{1,2\}$,
        \begin{enumerate}[label=(\alph*)]
            \item The function $k_\ell: [0,1] \to [m_{k_\ell}, M_{k_\ell}]$ is Lipschitz continuous, with $0< m_{k_\ell} \le M_{k_\ell} < \infty$. Consequently, the inverse permeability $k_\ell^{-1}: [0,1] \to [m_{\bar{k}_\ell},M_{\bar{k}_\ell}]$ is Lipschitz continuous, with $0< m_{\bar{k}_\ell} \le M_{\bar{k}_\ell} < \infty$ and Lipschitz constant $L_{{\bar k}_\ell}$.
            \item The function $S_\ell: \mathbb{R} \to [0,1]$ is monotonically increasing and Lipschitz continuous, with Lipschitz constant $L_{S_\ell}$.
        \end{enumerate}
    \end{assumption}
    
\subsection{Temporal discretization and linearization}
    Next, we derive the semi-discrete, linearized version of \eqref{eq:3}. For the sake of clarity, the interface conditions are omitted and will be treated in \Cref{sec: robin-robin}. First, we partition the time interval $[0,T]$ into uniform subintervals $N\in\mathbb{N}$, with $n=1,...,N$ being the time index, giving the corresponding time step $\tau  \coloneqq \frac{T}{N} $. Let $p_\ell^n, \mathbf{q}_\ell^n$ denote the pressure and flux at time $t^n \coloneqq \tau n$. Using the backward Euler method, \eqref{eq:3} becomes
    \begin{alignat*}{2}
        k_\ell^{-1}(S_\ell(p_\ell^{n})) \, \mathbf{q}_\ell^{n} + \nabla p_\ell^{n} &= \nabla \eta_\ell &\qquad\text{in } \Omega_\ell, \\
        S_\ell(p_\ell^{n}) - S_\ell(p_\ell^{n-1}) + \tau \nabla \cdot \mathbf{q}_\ell^{n} &= \tau \mathcal{F}_\ell &\qquad\text{in } \Omega_\ell, 
    \end{alignat*}
    for $\ell\in\{1,2\}$, where at time $t= 0$, we have $p_\ell^0 = {p_0}_{| \Omega_\ell}$.
    
    To handle the nonlinear dependency in the saturation, we linearize the mass balance equation using the L-scheme \cite{pop2004mixed, list2016study}. For the permeability, we lag the nonlinearity in the Darcy flow equation by one iteration. Thus, at iteration $i \in \mathbb{N}$, the permeability and saturation is evaluated at the previous iteration $i-1$. For readability, we introduce the shorthand notation 
    \begin{align}
        \bar{k}_\ell^{n,i-1} &\coloneqq k_\ell^{-1}(S_\ell(p_\ell^{n,i-1})), &
        S_\ell^{n,i-1} &\coloneqq S_\ell(p_\ell^{n,i-1}).
    \end{align}
    Introducing the stabilization parameter $L_\ell > 0$ for the L-scheme, the semi-discrete linearized system for each subdomain $\ell$ is given by
    \begin{subequations} \label{eq:6}
        \begin{alignat}{2}
            \bar{k}_\ell^{n,i-1} \, \mathbf{q}_\ell^{n,i} + \nabla p_\ell^{n,i} 
            &= \nabla \eta_\ell &\qquad\text{in } \Omega_\ell, \\
            S_\ell^{n,i-1} - S_\ell^{n-1} + L_\ell(p_\ell^{n,i} - p_\ell^{n,i-1}) + \tau \nabla \cdot \mathbf{q}_\ell^{n,i} 
            &= \tau \mathcal{F}_\ell &\qquad\text{in } \Omega_\ell. 
        \end{alignat}
    \end{subequations}
    for $\ell\in\{1,2\}$.

\subsection{Weak formulation}
    We continue with the introduction of the weak formulation of the linearized domain-decomposition scheme \eqref{eq:6}. For that, we first introduce the functional setting.
    Let $L^2(X)$ denote the space of square-integrable functions on $X \in\{\Omega, \Omega_\ell, \Gamma\}$, equipped with the inner product $\langle u,v \rangle_X \coloneqq \int_X u \ v \ \mathrm{d}x$ (with $\mathbf u \cdot \mathbf v$ replacing $uv$ for vector-valued functions) and induced norm $\lVert u \rVert_{L^2(X)} \coloneqq \langle u,u \rangle_X^{1/2}$. Let $\langle u,v \rangle_{\Omega_\ell}$ denote the inner product on $L^2(\Omega_\ell)$, $\ell \in \{1,2\}$, and $\lVert u \rVert_{\Omega_\ell}$ the corresponding norm. We omit the subscript to refer to the global $L^2$ inner product and norm
    \begin{align} \label{Notation1}
        D\langle u, v \rangle &\coloneqq \sum_{\ell=1}^2 D_\ell \langle u_\ell, v_\ell \rangle_{\Omega_\ell}, &
        D\lVert u \rVert^2 &\coloneqq \sum_{\ell=1}^2D_\ell \lVert u_\ell \rVert_{L^2(\Omega_\ell)}^2,
    \end{align}
    for some constant $D_\ell$. Similarly to inner products and norms, the omission of the subscript $\ell$ on any variable, e.g., $u$, denotes the global quantity defined over the entire domain $\Omega$. 
    
    Next, we introduce the standard Sobolev spaces. The Hilbert space for vector functions with square-integrable divergence is defined as
    \begin{equation*}
        H(\mathrm{div}; \Omega) \coloneqq \{ \tilde{\mathbf{q}} \in [L^2(\Omega)]^d \mid \nabla \cdot \tilde{\mathbf{q}} \in L^2(\Omega) \},
    \end{equation*}
    equipped with the norm
    \begin{equation} \label{eq:Hdiv norm}
        \lVert \tilde{\mathbf{q}} \rVert_{H(\mathrm{div}; \Omega)}^2 \coloneqq \lVert \tilde{\mathbf{q}} \rVert^2 + \lVert \nabla \cdot \tilde{\mathbf{q}} \rVert^2.
    \end{equation}
    For the flux and pressure variables, we consider the following function spaces defined on the subdomains
    \begin{align} \label{eq:local flux space}
        Q_\ell  &\coloneqq \{
            \tilde{\mathbf{q}}_\ell \in H(\mathrm{div}; \Omega_\ell) 
            \mid (\tilde{\mathbf{q}}_\ell \cdot \mathbf{n}_\ell)|_{\Gamma} \in L^2(\Gamma), \ 
            (\tilde{\mathbf{q}}_\ell \cdot \mathbf{n}_\ell)|_{\partial \Omega_N \cap \partial\Omega_\ell} =0 \}, \\
        P_\ell  &\coloneqq  L^2(\Omega_\ell).
    \end{align}
    and we equip $Q_\ell$ with the graph norm
    \begin{equation*}
        \lVert\tilde{\mathbf{q}}_\ell\rVert^2_{Q_\ell}  \coloneqq  \lVert\tilde{\mathbf{q}}_\ell\rVert_{\Omega_\ell}^2 + \lVert\nabla \cdot \tilde{\mathbf{q}}_\ell\rVert_{\Omega_\ell}^2 + \lVert\tilde{\mathbf{q}}_\ell \cdot \mathbf{n}_\ell\rVert^2_{L^2(\Gamma)}.
    \end{equation*}
    We, moreover, define the global product spaces
    \begin{equation*} 
        Q \coloneqq  Q_1 \times Q_2, \qquad
        P \coloneqq  P_1 \times P_2. \qquad
    \end{equation*}

    With the function spaces defined, we are ready to derive the weak formulation.
    Testing \eqref{eq:6} against $\tilde{\mathbf{q}}_\ell \in Q_\ell$ and $\tilde{p}_\ell \in P_\ell$, and integrating over $\Omega_\ell$,
    gives the intermediate form
    \begin{alignat*}{2}
        \langle \bar{k}_\ell^{n,i-1} \mathbf{q}_\ell^{n,i}, \tilde{\mathbf{q}}_\ell \rangle_{\Omega_\ell} 
        + \langle \nabla p_\ell^{n,i}, \tilde{\mathbf{q}}_\ell \rangle_{\Omega_\ell} 
        &= \langle \nabla \eta_\ell, \tilde{\mathbf{q}}_\ell \rangle_{\Omega_\ell}, \\
        \langle S_\ell^{n,i-1} - S_\ell^{n-1}, \tilde{p}_\ell \rangle_{\Omega_\ell}
        + L_\ell\langle p_\ell^{n,i} - p_\ell^{n,i-1}, \tilde{p}_\ell \rangle_{\Omega_\ell} 
        + \tau \langle \nabla \cdot \mathbf{q}_\ell^{n,i}, \tilde{p}_\ell \rangle_{\Omega_\ell} 
        &= \tau\langle \mathcal{F}_\ell, \tilde{p}_\ell \rangle_{\Omega_\ell}.
    \end{alignat*}
    Applying integration by parts, the boundary terms vanish due to the boundary conditions, and we are left with only the interface contributions. That is, 
    \begin{subequations} \label{eq:10} 
        \begin{alignat}{2}
            \langle \bar{k}_\ell^{n,i-1}\mathbf{q}_\ell^{n,i}, \tilde{\mathbf{q}}_\ell \rangle_{\Omega_\ell} 
            - \langle p_\ell^{n,i}, \nabla \cdot \tilde{\mathbf{q}}_\ell \rangle_{\Omega_\ell} 
            + \langle p_\ell^{n,i}, \tilde{\mathbf{q}}_\ell \cdot \mathbf{n}_\ell \rangle_\Gamma
            &= \langle \nabla \eta_\ell, \tilde{\mathbf{q}}_\ell \rangle_{\Omega_\ell}, \label{eq:10a}\\
            \langle S_\ell^{n,i-1} - S_\ell^{n-1}, \tilde{p}_\ell \rangle_{\Omega_\ell}
            + L_\ell\langle p_\ell^{n,i} - p_\ell^{n,i-1}, \tilde{p}_\ell \rangle_{\Omega_\ell} 
            + \tau \langle \nabla \cdot \mathbf{q}_\ell^{n,i}, \tilde{p}_\ell \rangle_{\Omega_\ell} 
            &= \tau\langle \mathcal{F}_\ell, \tilde{p}_\ell \rangle_{\Omega_\ell}. \label{eq:10b}
        \end{alignat}
    \end{subequations}

    \subsection{Reformulating the interface conditions as a Robin condition}
    \label{sec: robin-robin}
    As the pressure $p$ does not reside naturally in the trace space, we need a way to handle the interface terms in \eqref{eq:10a}. Instead of introducing some new interface variable, e.g., Lagrange multipliers $\lambda$, as an approximation to $p$, we now rewrite the pressure on the interface in terms of the normal components of the flux $\mathbf{q}$ based on the flux and pressure continuity across the interface $\Gamma$. In this way, the idea can be viewed as rewriting the Neumann boundary condition in \eqref{eq:3c} and the Dirichlet boundary condition in \eqref{eq:3d} into two Robin-type boundary conditions, one for each subdomain; see e.g.~\cite{seus2018linear}. With this in mind, instead of \eqref{eq:3c} and \eqref{eq:3d}, one can equivalently impose
    \begin{subequations} \label{eq:11}
        \begin{alignat}{1}
            p_1 - p_2  
            &= \alpha(\mathbf{q}_1 \cdot \mathbf{n}_1 
            + \mathbf{q}_2 \cdot \mathbf{n}_2), \label{eq:11a}\\
            p_2 - p_1
            &= \alpha(\mathbf{q}_1 \cdot \mathbf{n}_1 
            + \mathbf{q}_2 \cdot \mathbf{n}_2), \label{eq:11b}
        \end{alignat}
    \end{subequations}
    for some bounded $\alpha > 0$, where $\alpha$ is problem dependent. Clearly, if the flux and pressure continuities are satisfied, the equations above must hold. We can rewrite \eqref{eq:11} to obtain expressions for the pressure traces. A straightforward rearrangement yields
    \begin{alignat*}{1}
        p_1 &=\alpha \mathbf{q}_1 \cdot \mathbf{n}_1 
        + (\alpha \mathbf{q}_2 \cdot \mathbf{n}_2 + p_2), \\
        p_2 &= \alpha \mathbf{q}_2 \cdot \mathbf{n}_2 
        + (\alpha \mathbf{q}_1 \cdot \mathbf{n}_1 + p_1).
    \end{alignat*}
    Denoting the terms in parentheses as the Robin-variable 
    \begin{subequations} \label{eq:13}
        \begin{alignat}{1}
            g_1 &\coloneqq \alpha \mathbf{q}_2 \cdot \mathbf{n}_2 + p_2,\label{eq:13a}\\
            g_2 &\coloneqq \alpha \mathbf{q}_1 \cdot \mathbf{n}_1 + p_1, \label{eq:13b}
        \end{alignat}
    \end{subequations}
    we arrive at the expressions
    \begin{subequations} \label{eq:14}
        \begin{alignat}{1}
            p_1 &= \alpha \mathbf{q}_1 \cdot \mathbf{n}_1 + g_1, \label{eq:14a}\\
            p_2 &= \alpha \mathbf{q}_2 \cdot \mathbf{n}_2 + g_2. \label{eq:14b}
        \end{alignat}
    \end{subequations}

    Next, in analogy with the previous subsection, we introduce a superscript $n$ to denote the time step, i.e.,~$g^n = g(t^n)$. In order to evaluate $g^n$, we define $g_\ell^{n,i}$ in a lagged sense, i.e., based on the previous iterate. By inserting \eqref{eq:14} into \eqref{eq:13}, we get 
    \begin{subequations} \label{eq:15}
        \begin{alignat}{2}
            g_1^{n,i} & \coloneqq  \alpha \mathbf{q}_2^{n,i-1} \cdot \mathbf{n}_2 
            + (\alpha \mathbf{q}_2^{n,i-1} \cdot \mathbf{n}_2 + g_2^{n,i-1}) 
            &&= 2\alpha \mathbf{q}_2^{n,i-1} \cdot \mathbf{n}_2 + g_2^{n,i-1}, \label{eq:15a}\\
            g_2^{n,i} & \coloneqq  \alpha \mathbf{q}_1^{n,i-1} \cdot \mathbf{n}_1 
            + (\alpha \mathbf{q}_1^{n,i-1} \cdot \mathbf{n}_1 + g_1^{n,i-1}) 
            &&= 2\alpha \mathbf{q}_1^{n,i-1} \cdot \mathbf{n}_1 + g_1^{n,i-1}. \label{eq:15b}
        \end{alignat}
    \end{subequations}
    We emphasize that this update does not depend on the interface trace of the pressure. We finish this subsection with the function space for the interface variables, given by
    \begin{align}
        G &\coloneqq [L^2(\Gamma)]^2, &
        \langle g, \tilde g \rangle_G &\coloneqq \sum_{\ell=1}^2 \langle g_\ell, \tilde g_\ell \rangle_\Gamma, &
        \lVert g \rVert_G^2 &\coloneqq \sum_{\ell=1}^2 \lVert g_\ell \rVert_\Gamma^2.
    \end{align}
    Note that $G$ is the normal trace space of $Q$. 

\subsection{The iterative scheme}

    The final scheme is now obtained by using \eqref{eq:15} to replace the interface conditions \eqref{eq:3c} and \eqref{eq:3d}. Moreover, we use \eqref{eq:14} to substitute for $p_\ell^{n,i}$ in the interface term in \eqref{eq:10a}. This yields the following system,
    \begin{alignat*}{2}
        \langle \bar{k}_\ell^{n,i-1} \mathbf{q}_\ell^{n,i}, \tilde{\mathbf{q}}_\ell \rangle_{\Omega_\ell}
        - \langle p_\ell^{n,i}, \nabla \cdot \tilde{\mathbf{q}}_\ell \rangle_{\Omega_\ell}
        + \langle \alpha \mathbf{q}_\ell^{n,i} \cdot \mathbf{n}_\ell 
        + g_\ell^{n,i}, \tilde{\mathbf{q}}_\ell \cdot \mathbf{n}_\ell \rangle_\Gamma 
        &= \langle \nabla \eta_\ell, \tilde{\mathbf{q}}_\ell \rangle_{\Omega_\ell}, \\
        \langle S_\ell^{n,i-1} - S_\ell^{n-1}, \tilde{p}_\ell \rangle_{\Omega_\ell} 
        + L_\ell\langle p_\ell^{n,i} -p_\ell^{n,i-1}, \tilde{p}_\ell \rangle_{\Omega_\ell} 
        + \tau \langle \nabla \cdot \mathbf{q}_\ell^{n,i}, \tilde{p}_\ell \rangle_{\Omega_\ell} 
        &= \tau\langle \mathcal{F}_\ell, \tilde{p}_\ell \rangle_{\Omega_\ell}, \\
        \langle g_\ell^{n,i},\tilde{g}_\ell \rangle_\Gamma 
        &= \langle 2\alpha \mathbf{q}_\jay^{n,i-1} \cdot \mathbf{n}_\jay 
        + g_\jay^{n,i-1}, \tilde{g}_\ell \rangle_\Gamma,
    \end{alignat*}
    \text{for } $\ell,\jay \in \{1,2\}, \ell\neq \jay; \alpha\in(0,\infty)$.
    Moving all known terms to the right-hand side, we present the mLRDD-scheme, stated as the following problem:
    
    \begin{problem}[Iterative solution (mLRDD-scheme)] \label{FinalScheme}
    Given $\mathbf{q}^{n,i-1} \in Q$ and $g^{n,i - 1} \in G$, compute $g \in G$ such that
    \begin{subequations} \label{eq:final scheme}
        \begin{align}
            \langle g_\ell^{n,i},\tilde{g}_\ell \rangle_\Gamma
            &= \langle 2\alpha \mathbf{q}_\jay^{n,i-1} \cdot \mathbf{n}_\jay 
            + g_\jay^{n,i-1},\tilde{g}_\ell \rangle_\Gamma, &
            \ell, \jay &\in \{1, 2\}, \ \ell \ne \jay, \label{eq:final scheme c} 
        \end{align}
        for all $\tilde g \in G$. Then, given $p^{n,i-1} \in P$ and $S_\ell^{n - 1}$, find $\mathbf{q}^{n,i} \in Q$ and $p^{n,i} \in P$ such that
        \begin{alignat}{2}
            \langle \bar{k}_\ell^{n,i-1}\mathbf{q}_\ell^{n,i}, \tilde{\mathbf{q}}_\ell \rangle_{\Omega_\ell} 
            - \langle p_\ell^{n,i}, \nabla \cdot \tilde{\mathbf{q}}_\ell \rangle_{\Omega_\ell}  
            + \alpha \langle \mathbf{q}_\ell^{n,i} \cdot \mathbf{n}_\ell, \tilde{\mathbf{q}}_\ell \cdot \mathbf{n}_\ell \rangle_\Gamma 
            &= \langle \nabla \eta_\ell, \tilde{\mathbf{q}}_\ell \rangle_{\Omega_\ell} 
            - \langle g_\ell^{n,i}, \tilde{\mathbf{q}}_\ell \cdot \mathbf{n}_\ell \rangle_\Gamma, \label{eq:final scheme a}\\
            L_\ell\langle p_\ell^{n,i}, \tilde{p}_\ell \rangle_{\Omega_\ell}  
            + \tau \langle \nabla \cdot \mathbf{q}_\ell^{n,i}, \tilde{p}_\ell \rangle_{\Omega_\ell} 
            &= \tau\langle \mathcal{F}_\ell, \tilde{p}_\ell \rangle_{\Omega_\ell} 
            - \langle S_\ell^{n,i-1} - S_\ell^{n-1}, \tilde{p}_\ell \rangle_{\Omega_\ell} \nonumber\\
            &\quad + L_\ell\langle p_\ell^{n,i-1}, \tilde{p}_\ell \rangle_{\Omega_\ell} , \label{eq:final scheme b} 
        \end{alignat}
    \end{subequations}
    for all $\tilde{\mathbf{q}} \in Q$ and $\tilde{p} \in P$.
    \end{problem}

    We emphasize that Problem~\ref{FinalScheme} is a standard mixed formulation of Richards' equation with a storativity term. The bounds in Assumption~\ref{AssumptionOne} ensure that this problem is well-posed for each $i$. With the interest of convergence, we formally consider the limit problem as follows.

    \begin{problem}[Limit solution] \label{Limit problem}
    Given $S_\ell^{n-1}$, find $\mathbf{q}^n \in Q$, $p^n \in P$, and $g^n \in G$ such that
    \begin{subequations} \label{eq:limit solution}
        \begin{alignat}{1}
            \langle \bar{k}_\ell^{n}\mathbf{q}_\ell^{n}, \tilde{\mathbf{q}}_\ell \rangle_{\Omega_\ell}  
            - \langle p_\ell^{n}, \nabla \cdot \tilde{\mathbf{q}}_\ell \rangle_{\Omega_\ell}  
            + \alpha \langle \mathbf{q}_\ell^{n} \cdot \mathbf{n}_\ell, \tilde{\mathbf{q}}_\ell \cdot \mathbf{n}_\ell \rangle_\Gamma 
            &= \langle \nabla \eta_\ell, \tilde{\mathbf{q}}_\ell \rangle_{\Omega_\ell} 
            - \langle g_\ell^{n}, \tilde{\mathbf{q}}_\ell \cdot \mathbf{n}_\ell \rangle_\Gamma, \label{eq:limit solution a}\\
            \langle S_\ell^{n}, \tilde{p}_\ell \rangle_{\Omega_\ell}  
            + \tau \langle \nabla \cdot \mathbf{q}_\ell^{n}, \tilde{p}_\ell \rangle_{\Omega_\ell} 
            &= \tau\langle \mathcal{F}_\ell, \tilde{p}_\ell \rangle_{\Omega_\ell}  
            + \langle S_\ell^{n-1}, \tilde{p}_\ell \rangle_{\Omega_\ell} , \label{eq:limit solution b}\\
            \langle g_\ell^{n},\tilde{g}_\ell \rangle_\Gamma
            &= \langle  2\alpha \mathbf{q}_\jay^{n} \cdot \mathbf{n}_\jay 
            + g_\jay^{n}, \tilde{g}_\ell \rangle_\Gamma, \label{eq:limit solution c}
        \end{alignat}
    \end{subequations}
    with $\ell, \jay \in \{1, 2\}, \ \ell \ne \jay$, for all $\tilde{\mathbf{q}} \in Q$, $\tilde{p} \in P$, and $\tilde{g} \in G$.
    \end{problem}
    Note that \Cref{Limit problem} is equivalent to the weak form of the original system \eqref{eq:2}, under the assumption of sufficient regularity at the interfaces. Existence and uniqueness of a solution therefore follow from the well-posedness of the Richards flow problem \cite{alt1983quasilinear}. We end this section with a boundedness assumption on the limit solution.
\begin{assumption} \label{AssumptionTwo}
    The solution to Problem~\ref{Limit problem} satisfies $\sup_\ell \lVert\mathbf{q}_\ell^n \rVert_{L^\infty(\Omega_\ell)} \leq M_{\mathbf{q}_\ell} <+ \infty$. 
\end{assumption}

\section{Analysis of the convergence of the iterative scheme} \label{sec:analysis}
    In this section, we proceed to prove the convergence of the proposed Scheme \ref{FinalScheme} under Assumptions \ref{AssumptionOne} and \ref{AssumptionTwo}, with the appropriate parameter conditions, as presented in Theorem \ref{thm: convergence S q}.
    As the following proofs are carried out at a fixed time step $n$, we employ a slight abuse of notation by omitting the superscript $n$ from iteratively updated quantities. That is, 
    \begin{align*}
        p^{i}_\ell &\coloneqq p^{n,i}_\ell, & p^{i\pm1}_\ell &\coloneqq p^{n,i\pm1}_\ell,
    \end{align*}  
    and $\mathbf{q}^i_\ell$ and $g^i_\ell$ are defined analogously. The solution to the limit problem $p^n_\ell$, and the evaluation at the previous time step $p^{n-1}_\ell$, remain unchanged. 
    To further simplify notation, we define the errors at iteration $i$ and time step $n$ by
    \begin{equation}  \label{eq: errors}
        e_{p,\ell}^{i} \coloneqq p^{n}_\ell-p^{i}_\ell,
        \qquad e_{\mathbf{q},\ell}^{i}\coloneqq \mathbf{q}^{n}_\ell -\mathbf{q}^{i}_\ell,
        \qquad e_{g,\ell}^{i} \coloneqq g^{n}_\ell-g^{i}_\ell,
        \qquad e_{S,\ell}^{i} \coloneqq S^{n}_\ell-S^{i}_\ell.
    \end{equation}
    Similarly to the norm and inner product notations previously presented, the omission of the subscript $\ell$ on any variable, e.g., $e_p^i$, denotes the global quantity defined throughout the domain $\Omega$.
    
    \subsection{Subdomain variables}
    We start this section directly with our first main result, mainly the convergence of the flux and saturation. The pressure and interface variables are considered afterwards.
    \begin{theorem}\label{thm: convergence S q}
        Let Assumptions \ref{AssumptionOne} and \ref{AssumptionTwo} hold. If the linearization, interface, and time step parameters satisfy
        \begin{enumerate}[label=(\alph*)]
            \item $L_\ell > \frac12 L_{S_\ell}$,
            \item $\alpha > 0$,
            \item $\tau < (\frac{1}{L_{S_\ell}}- \frac{1}{2L_\ell})\frac{2 m_{\bar{k}_\ell}}{L_{\bar{k}_\ell}^2M_{\mathbf{q}_\ell}^2}$,
        \end{enumerate}
        for each $\ell \in \{1,2\}$, then the flux $\mathbf{q}^{i}$ and saturation $S^{i}$ converge strongly to $\mathbf{q}^{n}$ and $S^{n}$, respectively, in $L^2(\Omega)$. Specifically, 
        \begin{equation} \label{Thm1Conclusion}
            \lim_{i \to \infty} \lVert e_{\mathbf{q}}^i \rVert = 0 \qquad \textit{and} \qquad \lim_{i \to \infty} \lVert e_{S}^{i} \rVert = 0.
        \end{equation}\
    \end{theorem}
    \begin{proof}[Proof.]
    Following the initial approach in the proof of \cite{seus2018linear}, we begin by considering the error equations. We subtract the iterative scheme \eqref{eq:final scheme} from the limit problem \eqref{eq:limit solution}, and use the errors defined from \eqref{eq: errors}
    %
    \begin{subequations}
    \begin{alignat}{1}
        \langle \bar{k}_\ell^{n}\mathbf{q}_\ell^{n} - \bar{k}_\ell^{i-1}\mathbf{q}_\ell^{i}, \tilde{\mathbf{q}}_\ell \rangle_{\Omega_\ell} 
        - \langle e_{p_\ell}^{i}, \nabla \cdot \tilde{\mathbf{q}}_\ell \rangle_{\Omega_\ell} 
        + \alpha \langle e_{\mathbf{q},\ell}^{i} \cdot \mathbf{n}_\ell, \tilde{\mathbf{q}}_\ell \cdot \mathbf{n}_\ell \rangle_\Gamma
        + \langle e_{g,\ell}^{i}, \tilde{\mathbf{q}}_\ell \cdot \mathbf{n}_\ell \rangle_\Gamma
        = 0, \label{eq:beforesumming} \\    
        \tau \langle \nabla \cdot e_{\mathbf{q},\ell}^{i}, \tilde{p}_\ell \rangle_{\Omega_\ell}
        + L_\ell\langle p_\ell^{i-1} - p_\ell^{i}, \tilde{p}_\ell \rangle_{\Omega_\ell}
        + \langle e_{S,\ell}^{i}, \tilde{p}_\ell \rangle_{\Omega_\ell}
        = 0.
    \end{alignat}
    \end{subequations}
    Next, we sum over all subdomains and employ the shorthand notation from \eqref{Notation1}
    \begin{subequations} \label{eq:22}
        \begin{alignat}{1}
            \langle \bar{k}^{n}\mathbf{q}^{n} - \bar{k}^{i-1}\mathbf{q}^{i}, \tilde{\mathbf{q}} \rangle 
            - \langle e_{p}^{i}, \nabla \cdot \tilde{\mathbf{q}} \rangle 
            + \alpha \langle e_{\mathbf{q}}^{i} \cdot \mathbf{n}, \tilde{\mathbf{q}} \cdot \mathbf{n} \rangle_G
            + \langle e_{g}^{i}, \tilde{\mathbf{q}} \cdot \mathbf{n} \rangle_G
            = 0, \label{eq:22a} \\    
            \tau \langle \nabla \cdot e_{\mathbf{q}}^{i}, \tilde{p} \rangle
            + L \langle (p^{i-1} - p^{i}), \tilde{p} \rangle
            + \langle e_{S}^{i}, \tilde{p} \rangle
            = 0. \label{eq:22b}
        \end{alignat}
    \end{subequations}
    We emphasize that the omission of the subscript $\ell$, as described in \eqref{Notation1}, includes subdomain-dependent constants. Add and subtract $\langle \bar{k}^{i-1}\mathbf{q}^{n}, \tilde{\mathbf{q}} \rangle$ and $L\langle p^{n}, \tilde{p} \rangle$ to the first and second equation, respectively, to obtain 
    \begin{alignat*}{1}
        \langle (\bar{k}^{n} - \bar{k}^{i-1}) \mathbf{q}^{n}, \tilde{\mathbf{q}} \rangle
        + \langle \bar{k}^{i-1} e_{\mathbf{q}}^i, \tilde{\mathbf{q}} \rangle
        - \langle e_{p}^{i}, \nabla \cdot \tilde{\mathbf{q}} \rangle 
        + \alpha \langle e_{\mathbf{q}}^{i} \cdot \mathbf{n}, \tilde{\mathbf{q}} \cdot \mathbf{n} \rangle_G
        + \langle e_{g}^{i}, \tilde{\mathbf{q}} \cdot \mathbf{n} \rangle_G 
        = 0, \\
        L\langle e_{p}^{i}, \tilde{p} \rangle 
        - L\langle e_{p}^{i-1}, \tilde{p}\rangle
        + \tau \langle \nabla \cdot e_{\mathbf{q}}^{i}, \tilde{p} \rangle
        +\langle e^{i-1}_{S} , \tilde{p} \rangle
        = 0.
    \end{alignat*}
    We now test the above equations with $\tilde{\mathbf{q}}=\tau e_{\mathbf{q}}^{i}$ and $\tilde{p} = e_{p}^{i}$, respectively, yielding
    \begin{alignat*}{1}
        \tau \langle (\bar{k}^{n} - \bar{k}^{i-1}) \mathbf{q}^{n},  e_{\mathbf{q}}^{i} \rangle
         - \tau \langle \bar{k}^{i-1} e_{\mathbf{q}}^i,  e_{\mathbf{q}}^{i} \rangle
        - \tau \langle e_{p}^{i}, \nabla \cdot e_{\mathbf{q}}^{i} \rangle 
        + \tau \alpha \langle e_{\mathbf{q}}^{i} \cdot \mathbf{n}, e_{\mathbf{q}}^{i} \cdot \mathbf{n} \rangle_G
         + \tau \langle e_{g}^{i}, e_{\mathbf{q}}^{i} \cdot \mathbf{n} \rangle_G 
        = 0, \\
        L\langle e_{p}^{i}, e_{p}^{i} \rangle 
        -L\langle e_{p}^{i-1}, e_{p}^{i} \rangle 
        + \tau \langle \nabla \cdot e_{\mathbf{q}}^{i}, e_{p}^{i} \rangle
        + \langle e^{i-1}_{S} , e_{p}^{i} \rangle
        = 0.
    \end{alignat*}
    Further, we add and subtract $\langle e^{i-1}_{S}, e_{p}^{i-1} \rangle$ to the last equation above to obtain
    \begin{alignat*}{1}
      \tau \langle (\bar{k}^{n} - \bar{k}^{i-1}) \mathbf{q}^{n}, e_{\mathbf{q}}^{i} \rangle
      + \tau \langle \bar{k}^{i-1} e_{\mathbf{q}}^{i}, e_{\mathbf{q}}^{i} \rangle
      - \tau \langle e_{p}^{i}, \nabla \cdot e_{\mathbf{q}}^{i} \rangle 
      + \tau \alpha \langle e_{\mathbf{q}}^{i} \cdot \mathbf{n}, e_{\mathbf{q}}^{i} \cdot \mathbf{n} \rangle_G
      + \tau \langle e_{g}^{i}, e_{\mathbf{q}}^{i} \cdot \mathbf{n} \rangle_G
      = 0, \\
      L \langle e_{p}^{i}, e_{p}^{i} \rangle
      - L \langle e_{p}^{i-1}, e_{p}^{i} \rangle
      + \tau \langle \nabla \cdot e_{\mathbf{q}}^{i}, e_{p}^{i} \rangle
      + \langle e^{i-1}_{S}, e_{p}^{i} - e_{p}^{i-1} \rangle
      + \langle e^{i-1}_{S}, e_{p}^{i-1} \rangle
      = 0.
      \end{alignat*}
    Now, adding together the mass balance and Darcy flow error equation, the divergence terms cancel, giving
    \begin{equation*} 
        \begin{aligned}
            \tau \langle (\bar{k}^{n} &- \bar{k}^{i-1})\mathbf{q}^{n}, e_{\mathbf{q}}^{i} \rangle 
            + \tau \langle \bar{k}^{i-1} e_{\mathbf{q}}^{i}, e_{\mathbf{q}}^{i} \rangle
            + L \langle e_{p}^{i} - e_{p}^{i-1}, e_{p}^{i} \rangle
            + \langle e^{i-1}_{S} , e_{p}^{i} - e_{p}^{i-1} \rangle \\
            &+ \langle e^{i-1}_{S} , e_{p}^{i-1} \rangle 
            + \tau \alpha \langle e_{\mathbf{q}}^{i} \cdot \mathbf{n}, e_{\mathbf{q}}^{i} \cdot \mathbf{n} \rangle_G 
            +\tau \langle e_{g}^{i}, e_{\mathbf{q}}^{i} \cdot \mathbf{n} \rangle_G
            = 0.
        \end{aligned}
    \end{equation*}
    After rearranging the above terms, we arrive at 
    \begin{equation}
        \begin{aligned} \label{eq:final error equation}
        \overset{\mathrm{I}}{\left[\tau \langle \bar{k}^{i-1} e_{\mathbf{q}}^{i}, e_{\mathbf{q}}^{i} \rangle \right]}
        &+ \overset{\mathrm{II}}{\left[L \langle e_{p}^{i} - e_{p}^{i-1}, e_{p}^{i} \rangle\right]}
        + \overset{\mathrm{III}}{\left[\langle e^{i-1}_{S} , e_{p}^{i-1} \rangle \right]}
        + \overset{\mathrm{IV}}{\left[\tau \alpha \langle e_{\mathbf{q}}^{i} \cdot \mathbf{n}, e_{\mathbf{q}}^{i} \cdot \mathbf{n} \rangle_G
        +\tau \langle e_{g}^{i}, e_{\mathbf{q}}^{i} \cdot \mathbf{n} \rangle_G \right]} \\
        &= \overset{\mathrm{V}}{\left[- \tau \langle (\bar{k}^{n} - \bar{k}^{i-1})\mathbf{q}^{n}, e_{\mathbf{q}}^{i} \rangle \right]}
        +\overset{\mathrm{VI}}{\left[- \langle e^{i-1}_{S} , e_{p}^{i} - e_{p}^{i-1} \rangle \right]}.
        \end{aligned}
    \end{equation}

    We now proceed by bounding each of the terms in \eqref{eq:final error equation}. First, we formulate lower bounds for the left-hand side terms, followed by upper bounds for the right-hand side.
    \begin{enumerate}[label=\Roman*.]
        \item Starting with the left-hand side, we use the lower bound of $\bar{k}$ from Assumption~\ref{AssumptionOne} to bound the first term by
        \begin{equation}
            \tau \langle \bar{k}^{i-1} e_{\mathbf{q}}^{i}, e_{\mathbf{q}}^{i} \rangle \geq \tau m_{\bar{k}} \lVert e_{\mathbf{q}}^{i} \rVert^2.
        \end{equation}
        \item Using the identity $(x-y) x = (x^2 - y^2 + (x-y)^2)/2$ on the second term, we have
        \begin{align}
            L \langle e_{p}^{i} - e_{p}^{i-1}, e_{p}^{i} \rangle
            =
            \frac{L}{2}&\left(\lVert e_{p}^{i} \rVert^2 
            - \lVert e_{p}^{i-1} \rVert^2
            + \lVert e_{p}^{i}- e_{p}^{i-1} \rVert^2\right).
        \end{align}
        \item The monotonicity and Lipschitz continuity of $S$, cf.~Assumption~\ref{AssumptionOne}, bounds the third term as
        \begin{equation*}
            \langle e^{i-1}_{S}, e_{p}^{i-1} \rangle 
            \geq \frac{1}{L_{S}}\lVert e^{i-1}_{S} \rVert^2.
        \end{equation*}
        \item For the two interface terms, we first derive an identity. By subtracting \eqref{eq:final scheme c} from \eqref{eq:limit solution c}, we arrive at
        \begin{equation*}
        \langle g_\ell^{n}- g_\ell^{i},\tilde{g}_\ell \rangle_\Gamma
        = \langle  2\alpha (\mathbf{q}_\jay^{n} \cdot \mathbf{n}_\jay
        - \mathbf{q}_\jay^{i-1} \cdot \mathbf{n}_\jay) 
        + g_\jay^{n} -g_\jay^{i-1}, \tilde{g}_\ell \rangle_\Gamma.
        \end{equation*}
        Using the defined errors in \eqref{eq: errors}, shifting by one iteration and considering the equation from the adjacent subdomain $\Omega_\jay$, we get
        \begin{equation} \label{eq:34}
            \langle e_{g,\jay}^{i+1},\tilde{g}_\jay \rangle_\Gamma
            = \langle  2\alpha \ e_{\mathbf{q},\ell}^{i} \cdot \mathbf{n}_\ell 
            + e_{g,\ell}^{i}, \tilde{g}_\jay \rangle_\Gamma.
        \end{equation}
        As the above holds for all $\tilde{g} \in G$, we apply the polarization identity $\lVert x+y \rVert^2 = \lVert x \rVert^2 + \lVert y \rVert^2 + 2\langle x,y\rangle$ to derive
        \begin{equation*}
            \lVert e_{g,\jay}^{i+1} \rVert^2_\Gamma
            = \lVert 2\alpha \ e_{\mathbf{q},\ell}^{i} \cdot \mathbf{n}_\ell 
            + e_{g,\ell}^{i} \rVert^2_\Gamma
            = 4\alpha^2 \lVert e_{\mathbf{q},\ell}^{i} \cdot \mathbf{n}_\ell \rVert^2_\Gamma 
            + \lVert e_{g,\ell}^{i} \rVert^2_\Gamma 
            + 4\alpha\langle e_{g,\ell}^{i}, e_{\mathbf{q},\ell}^{i} \cdot \mathbf{n}_\ell\rangle_\Gamma.
        \end{equation*}
        Rearranging and summing over the two subdomains, we obtain
        \begin{equation*}
            \lVert e_{\mathbf{q}}^{i} \cdot \mathbf{n} \rVert^2_G 
            = \frac{1}{4\alpha^2}
            \left(\lVert e_{g}^{i+1} \rVert^2_G 
            - \lVert e_{g}^{i} \rVert^2_G
            - 4\alpha\langle e_{g}^{i}, e_{\mathbf{q}}^{i} \cdot \mathbf{n}\rangle_G \right).
        \end{equation*}
        The above identity is then applied to the last left-hand side's fourth term, to arrive at 
        \begin{equation}
            \begin{aligned}
                \tau \alpha \langle e_{\mathbf{q}}^{i} \cdot \mathbf{n}, e_{\mathbf{q}}^{i} \cdot \mathbf{n} \rangle_G 
                + \tau \langle e_{g}^{i}, e_{\mathbf{q}}^{i} \cdot \mathbf{n} \rangle_G
                &= \tau \alpha \lVert e_{\mathbf{q}}^{i} \cdot \mathbf{n} \rVert^2_G
                +\tau \langle e_{g}^{i}, e_{\mathbf{q}}^{i} \cdot \mathbf{n} \rangle_G \\
                &= \frac{\tau}{4\alpha} \left(\lVert e_{g}^{i+1} \rVert^2_G
                - \lVert e_{g}^{i} \rVert^2_G  
                - 4\alpha\langle e_{g}^{i}, e_{\mathbf{q}}^{i} \cdot \mathbf{n} \rangle_G \right)
                +\tau \langle e_{g}^{i}, e_{\mathbf{q}}^{i} \cdot \mathbf{n} \rangle_G \\
                &= \frac{\tau}{4\alpha} (\lVert e_{g}^{i+1} \rVert^2_G
                - \lVert e_{g}^{i} \rVert^2_G ).
            \end{aligned}
        \end{equation}
        \item \label{bound V} Next, we consider the right-hand side. The first term is bounded by using Cauchy-Schwarz, the Lipschitz continuity of $\bar{k}$ from Assumption~\ref{AssumptionOne}, the boundedness of $\mathbf{q}^n$ from Assumption~\ref{AssumptionTwo}, and Young's inequality, yielding
        \begin{align*} 
            -\tau \langle (\bar{k}^{n}-\bar{k}^{i-1})\mathbf{q}^{n}, e_{\mathbf{q}}^{i} \rangle
            &\leq \tau \lVert (\bar{k}^{n}-\bar{k}^{i-1}) \mathbf{q}^{n} \rVert \lVert e_{\mathbf{q}}^{i} \rVert\\ 
            &\leq \tau L_{\bar{k}} \lVert S^{n}-S^{i-1} \rVert \lVert \mathbf{q}^{n} \rVert_{L^\infty(\Omega)} \lVert e_{\mathbf{q}}^{i} \rVert\\
            &\leq \tau L_{\bar{k}} M_{\mathbf{q}} \lVert e^{i-1}_{S} \rVert \lVert e_{\mathbf{q}}^{i} \rVert\\
            &\leq \tau L_{\bar{k}} M_{\mathbf{q}} \delta \lVert e^{i-1}_{S} \rVert^2 + \frac{\tau L_{\bar{k}}M_{\mathbf{q}}}{ 4\delta} \lVert e_{\mathbf{q}}^{i} \rVert^2,
        \end{align*}
        with $\delta > 0$ specified later.
        \item Finally, we use the Cauchy-Schwarz and Young's inequalities to bound the last right-hand side term by
    \begin{equation*}
        - \langle e^{i-1}_{S} , e_{p}^{i} - e_{p}^{i-1} \rangle
        \leq \frac{1}{2L} \lVert e^{i-1}_{S} \rVert^2 + \frac{L}{2}\lVert e_{p}^{i} - e_{p}^{i-1} \rVert^2.
    \end{equation*}
\end{enumerate}
    With the estimates from I-VI, \eqref{eq:final error equation} becomes
    \begin{equation*}
        \begin{aligned}
            \tau m_{\bar{k}} \lVert e_{\mathbf{q}}^{i} \rVert^2
            &+ \frac{L}{2}\left(\lVert e_{p}^{i} \rVert^2 
            - \lVert e_{p}^{i-1} \rVert^2
            + \lVert e_{p}^{i}- e_{p}^{i-1} \rVert^2\right)
            + \frac{1}{L_{S}}\lVert e^{i-1}_{S} \rVert^2
            + \frac{\tau}{4\alpha} (\lVert e_{g}^{i+1} \rVert^2_G
            - \lVert e_{g}^{i} \rVert^2_G ) \\
            &\leq 
            \tau L_{\bar{k}} M_{\mathbf{q}} \delta \lVert e^{i-1}_{S} \rVert^2 + \frac{\tau L_{\bar{k}}M_{\mathbf{q}}}{ 4\delta} \lVert e_{\mathbf{q}}^{i} \rVert^2
            + \frac{1}{2L} \lVert e^{i-1}_{S} \rVert^2 + \frac{L}{2}\lVert e_{p}^{i} - e_{p}^{i-1} \rVert^2.
        \end{aligned}
    \end{equation*}
    Moving all the terms to the left-hand side and simplifying the expression, we arrive at
    \begin{equation*}
        \begin{aligned}
            \tau(m_{\bar{k}} &- \frac{L_{\bar{k}}M_{\mathbf{q}}}{ 4\delta}) \lVert e_{\mathbf{q}}^{i} \rVert^2
            + \frac{L}{2}(\lVert e_{p}^{i} \rVert^2
            - \lVert e_{p}^{i-1} \rVert^2)
            + (\frac{1}{L_{S}} - \frac{1}{2L}-\tau L_{\bar{k}} M_{\mathbf{q}} \delta ) \lVert e^{i-1}_{S} \rVert^2
            + \frac{\tau}{4\alpha} ( \lVert e_{g}^{i+1} \rVert^2_G
            - \lVert e_{g}^{i} \rVert^2_G )
            \leq 0.
        \end{aligned}
    \end{equation*}
    Next, we choose the parameter of Young's inequality from term \hyperref[bound V]{V} to be $\delta = \frac{L_{\bar{k}} M_{\mathbf{q}}}{2 m_{\bar{k}}}$. Using the assumption on $L > \frac{L_S}{2}$ and $\tau < \left(\frac{1}{L_{S}}- \frac{1}{2L}\right)\frac{2 m_{\bar{k}}}{L_{\bar{k}}^2M_{\mathbf{q}}^2}$, we obtain
    \begin{align}
        C_1 &\coloneqq \tau \left(m_{\bar{k}} - \frac{L_{\bar{k}}M_{\mathbf{q}}}{ 4\delta}\right) = \tau\frac{m_{\bar{k}}}{2} > 0, \\
        C_2 &\coloneqq \left(\frac{1}{L_{S}}-\frac{1}{2L}\right)-\tau L_{\bar{k}} M_{\mathbf{q}} \delta = \left(\frac{1}{L_{S}}-\frac{1}{2L}\right)-\tau \frac{L_{\bar{k}}^2 M_{\mathbf{q}}^2}{2 m_{\bar{k}}}> 0.
    \end{align}
    With these definitions, the inequality simplifies. By summing over all iteration indices $i=1,...,I$, we have
    \begin{equation*}
        \begin{aligned}
            \sum_{i = 1}^{I} C_1 \lVert e_{\mathbf{q}}^{i} \rVert^2
            + \sum_{i = 1}^{I}\frac{L}{2}( \lVert e_{p}^{i} \rVert^2
            - \lVert e_{p}^{i-1} \rVert^2)
            + \sum_{i = 1}^{I}C_2 \lVert e^{i-1}_{S} \rVert^2
            + \sum_{i = 1}^{I}\frac{\tau}{4\alpha} ( \lVert e_{g}^{i+1} \rVert^2_G
            - \lVert e_{g}^{i} \rVert^2_G ) \leq 0.
        \end{aligned}
    \end{equation*} 
    From the telescoping property of the above inequality, we get
    \begin{equation} \label{eq:40}
        \begin{aligned}
            \sum_{i = 1}^{I} C_1 \lVert e_{\mathbf{q}}^{i} \rVert^2
            + \frac{L}{2}( \lVert e_{p}^{I} \rVert^2
            - \lVert e_{p}^{0} \rVert^2) 
            + \sum_{i = 1}^{I}C_2 \lVert e^{i-1}_{S} \rVert^2 
            + \frac{\tau}{4\alpha} ( \lVert e_{g}^{I+1} \rVert^2_G
            - \lVert e_{g}^{1} \rVert^2_G )
            \leq 0.
        \end{aligned}
    \end{equation} 
    A rearrangement now yields
    \begin{equation} \label{eq:41}
        \sum_{i = 1}^{I}  C_1 \lVert e_{\mathbf{q}}^{i} \rVert^2
        + \sum_{i = 1}^{I} C_2 \lVert e^{i-1}_{S} \rVert^2 
        + \frac{L}{2} \lVert e_{p}^{I} \rVert^2
        + \frac{\tau}{4\alpha} \lVert e_{g}^{I+1} \rVert^2_G
        \leq \frac{L}{2} \lVert e_{p}^{0} \rVert^2 + \frac{\tau}{4\alpha} \lVert e_{g}^{1} \rVert^2_G.
    \end{equation}
    The right hand side is independent of $I \in \mathbb{N} $, so as we let $I \to \infty$, we get the estimate
    \begin{alignat*}{1}
        \sum_{i = 1}^{\infty}  C_1 \lVert e_{\mathbf{q}}^{i} \rVert^2
        + \sum_{i = 1}^{\infty} C_2 \lVert e^{i-1}_{S} \rVert^2 
        \leq \frac{L}{2} \lVert e_{p}^{0} \rVert^2 + \frac{\tau}{4\alpha} \lVert e_{g}^{1} \rVert^2_G.
    \end{alignat*}
    This means that the infinite sum on the left-hand side is bounded by the finite errors at the first iterations, and we conclude \eqref{Thm1Conclusion}.
    \end{proof}
    
    Next, we turn our attention to the convergence of the pressure $p^{i}$ and the interface variable $g^{i}$. We start by deriving a weak convergence result using the following lemma.
    \begin{lemma}
        Let $X$ be a Banach space with dual $X^*$, let $\{x^i\}_{i} \subset X$, and let $x \in X$. Suppose every subsequence of $\{x^i\}_{i}$ has a further sub-subsequence converging weakly to $x$. Then
        \begin{equation*}
            x^i \rightharpoonup x.
        \end{equation*}
    \end{lemma}
        
    \begin{proof}
        We argue by contradiction. Suppose that $x^i \not\rightharpoonup x$. Then there exists a functional $\varphi \in X^*$, $\varepsilon_0 > 0$ and a subsequence $\{x^{i_k}\}_k$ of $\{x^i\}_{i}$ such that
        \begin{equation}\label{eq:subseq}
            \left|\varphi(x^{i_k}) - \varphi(x)\right| \geq \varepsilon_O \qquad \forall k \in \mathbb{N}.
        \end{equation}
        By hypothesis, $\{x^{i_k}\}_k$ has a further sub-subsequence
        $\{x^{i_{k_l}}\}_l$ with $x^{i_{k_l}} \rightharpoonup x$. In particular, $\varphi(x^{i_{k_l}}) \to \varphi(x)$,
        but this contradicts \eqref{eq:subseq}, since every term of $\{x^{i_{k_l}}\}$ is a term of $\{x^{i_k}\}$.
    \end{proof}
    
    \begin{lemma} \label{lem: weak convergence p and g}
        The pressure $p^{i}$ and the Robin-variable $g^{i}$ converge weakly
        \begin{align}
            \lim_{i \to \infty} \langle e_g^i, \tilde g \rangle_G &= 0, 
            &\forall \tilde g &\in G, \\
            \lim_{i \to \infty} \langle e_p^i, \tilde p \rangle &= 0, 
            &\forall \tilde p &\in P.
        \end{align}
    \end{lemma}
    \begin{proof}
        From estimate \eqref{eq:41}, we extract the following error bound
        \begin{align}
            \frac{L}{2} \lVert e_{p}^{I} \rVert^2
            + \frac{\tau}{4\alpha} \lVert e_{g}^{I+1} \rVert^2_G 
            &\leq \frac{L}{2} \lVert e_{p}^{0} \rVert^2 + \frac{\tau}{4\alpha} \lVert e_{g}^{1} \rVert^2_G, &
            \forall I &\in \mathbb{N}.
    \end{align}
        Consequently, the sequences $\{e_{p}^i\}_{i \in \mathbb{N}}$ and $\{e_{g}^i\}_{i \in \mathbb{N}}$ are uniformly bounded in $P$ and $G$, respectively. To derive the results, we set our focus on the Robin interface variable. Since $\{e_{g}^i\}_{i \in \mathbb{N}}$ is bounded, every subsequence has a weakly convergent sub-subsequence \cite[D Thm.~3]{evans2022partial}. Let $e_g^*$ be such a weak limit. Passing to the limit, we see that $g^* \coloneqq g^n - e_g^*$ satisfies \eqref{eq:limit solution}. However, $g^n$ is the unique solution to \eqref{eq:limit solution} by Assumption~\ref{AssumptionTwo}, so $e_g^* = 0$. Thus, every subsequence has a sub-subsequence that converges weakly to zero, and \Cref{lem: weak convergence p and g} then implies that the entire sequence converges weakly to zero. The same argument applies for the sequence of pressure errors.
    \end{proof}

    \begin{rmk}
        Lemma~\ref{lem: weak convergence p and g} suffices in the discrete case because, in the finite-dimensional setting, weak convergence does imply strong convergence and all norms are equivalent. 
    \end{rmk}

    We now improve on this result by considering strong convergence of the pressure. For that, we first require the inf-sup condition given in the following lemma.
    \begin{lemma}\label{lemma1}
        Given that $\Omega$ is subdivided into two Lipschitz domains $\Omega_\ell$, then for any $\tilde{f} \in P$ there exists a $\tilde{\mathbf{q}} \in Q$ such that
        \begin{equation*} \label{eq:42}
        \nabla\cdot\tilde{\mathbf{q}}=\tilde{f} \quad and \quad \lVert \tilde{\mathbf{q}} \rVert\leq C_\Omega\lVert \tilde{f} \rVert, \quad with \quad \tilde{\mathbf{q}} \cdot \mathbf{n}|_\Gamma= 0,
        \end{equation*}
        where the constant $C_\Omega$ is independent of $\tilde{f}$.
    \end{lemma}    
    \begin{proof}
        See \cite{thomas1977sur} for a construction, which we apply to each $\Omega_\ell$. This is possible due to the assumption that $\partial\Omega_\ell \cap \partial \Omega_D \ne \emptyset$, cf.~\Cref{sec:problem}.
    \end{proof}
    \begin{theorem} \label{thm: convergence pressure}
        The pressure $p^{i}$ converges strongly to $p^{n}$ in $L^2(\Omega)$. That is,
        \begin{equation*}
            \lim_{i \to \infty} \lVert e_{p}^i \rVert = 0.
        \end{equation*}
    \end{theorem}
    \begin{proof}
        We bound the pressure error term using the inf-sup condition from the previous Lemma \ref{lemma1}. Starting from \eqref{eq:22a} with an arbitrary test function, we have
        \begin{equation*}
            \langle \bar{k}^{n}\mathbf{q}^{n} - \bar{k}^{i-1}\mathbf{q}^{i}, \tilde{\mathbf{q}} \rangle 
            - \langle e_{p}^{i}, \nabla \cdot \tilde{\mathbf{q}} \rangle 
            + \alpha \langle e_{\mathbf{q}}^{i} \cdot \mathbf{n}, \tilde{\mathbf{q}} \cdot \mathbf{n} \rangle_G
            + \langle e_{g}^{i}, \tilde{\mathbf{q}} \cdot \mathbf{n} \rangle_G
            = 0.
        \end{equation*}
        From Lemma \ref{lemma1}, by choosing $\tilde{f} = e_{p}^{i}$, there exists a $\tilde{\mathbf{q}} \in Q$ such that $\nabla\cdot\tilde{\mathbf{q}} = e_{p}^{i}$ and $\lVert \tilde{\mathbf{q}}\rVert\leq C_\Omega \lVert e_{p}^{i}\rVert$ with $\tilde{\mathbf{q}} \cdot \mathbf{n}|_\Gamma = 0$, giving us
        \begin{equation*}
            \langle \bar{k}^{n}\mathbf{q}^{n} - \bar{k}^{i-1}\mathbf{q}^{i}, \tilde{\mathbf{q}} \rangle
            - \langle e_{p}^{i}, e_{p}^{i} \rangle 
            = 0.
        \end{equation*}
        Following the steps in the proof of \Cref{thm: convergence S q}, we add and subtract $\langle \bar{k}^{i-1}\mathbf{q}^{n}, \tilde{\mathbf{q}} \rangle$, giving
        \begin{equation*}
            \lVert e_{p}^{i} \rVert^2 
            = \langle (\bar{k}^{n}-\bar{k}^{i-1})\mathbf{q}^{n}, \tilde{\mathbf{q}} \rangle 
            + \langle \bar{k}^{i-1} e_{\mathbf{q}}^{i}, \tilde{\mathbf{q}} \rangle.
        \end{equation*}
        Using the Cauchy-Schwarz inequality, the Lipschitz continuity and upper bound for $\bar{k}$, and the boundedness of $\mathbf{q}$, we have
        \begin{equation*}
            \lVert e_{p}^{i} \rVert^2 
            \leq L_{\bar{k}} M_{\mathbf{q}} \lVert e^{i-1}_{S} \rVert
            \lVert\tilde{\mathbf{q}} \rVert
            +  M_{\bar{k}}\lVert e_{\mathbf{q}}^{i} \rVert \lVert \tilde{\mathbf{q}} \rVert,
        \end{equation*}
        and, by the bound on $\tilde{\mathbf{q}}$ from Lemma \ref{lemma1},
        \begin{equation*}
            \lVert e_{p}^{i} \rVert^2 
            \leq C_\Omega ( L_{\bar{k}} M_{\mathbf{q}} \lVert e^{i-1}_{S} \rVert
            + M_{\bar{k}} \lVert e_{\mathbf{q}}^{i} \rVert) \lVert e_{p}^{i} \rVert.
        \end{equation*} 
        Thus, we have that
        \begin{equation*}
            \lVert e_{p}^{i} \rVert
            \leq C_\Omega ( L_{\bar{k}} M_{\mathbf{q}} \lVert e^{i-1}_{S} \rVert
            + M_{\bar{k}} \lVert e_{\mathbf{q}}^{i} \rVert),
        \end{equation*} 
        and the result follows from the results of \Cref{thm: convergence S q}.
    \end{proof}
    
    Next, we consider the divergence of the flux. By showing the convergence of the pressure, it allows us to prove that the divergence of the flux converges strongly as well.
    \begin{theorem}\label{theorem3}
        The divergence of the flux $\nabla \cdot \mathbf{q}^{i}$ converges strongly to $\nabla \cdot \mathbf{q}^{n}$. Specifically, 
        \begin{equation*}
            \lim_{i \to \infty} \lVert\nabla \cdot e_{\mathbf{q}}^i \rVert = 0.
        \end{equation*}
    \end{theorem}
    \begin{proof}
        Starting from \eqref{eq:22b}, we have
        \begin{equation*}
            L\langle e_{p}^{i}, \tilde{p} \rangle 
            - L \langle e_{p}^{i-1}, \tilde{p} \rangle
            + \tau \langle \nabla \cdot e_{\mathbf{q}}^{i}, \tilde{p} \rangle
            +\langle e^{i-1}_{S} , \tilde{p} \rangle
            = 0.
        \end{equation*}
        Since $\mathbf{q}^{i} \in Q$, we have that $\nabla \cdot \mathbf{q}^{i} \in P$. We may therefore test this equation with $\tilde{p} = \nabla \cdot e_{\mathbf{q}}^{i}$ to obtain
        \begin{equation*}
            \tau \lVert \nabla \cdot e_{\mathbf{q}}^{i}\rVert^2
            = L\langle e_{p}^{i-1} - e_{p}^{i}, \nabla \cdot e_{\mathbf{q}}^{i} \rangle
            -\langle e^{i-1}_{S}, \nabla \cdot e_{\mathbf{q}}^{i} \rangle.
        \end{equation*}
        Using the Cauchy-Schwarz inequality, we get
        \begin{equation*}
            \lVert\nabla \cdot e_{\mathbf{q}}^{i}\rVert
            \leq \frac{L}{\tau}\lVert e_{p}^{i-1} - e_{p}^{i} \rVert + \frac{1}{\tau} \lVert e^{i-1}_{S} \rVert,
        \end{equation*}
        and \Cref{thm: convergence S q,thm: convergence pressure} now yield the desired result.
    \end{proof}

    The previous theorem implies that mass is conserved locally in each subdomain. The mass balance across the interface $\Gamma$ is considered in the next subsection. 
   
    \subsection{Interface variables}
    
    In this section, we look at the convergence of the interface variables. For the proofs to come, we need the following norms and function spaces. We start with the Hilbert space of scalar functions with square-integrable weak derivatives, and its subspace with vanishing trace on $\partial\Omega$, respectively, defined as
    \begin{equation*}
        H^1(\Omega) \coloneqq \{ \phi \in L^2(\Omega) \mid \nabla \phi \in [L^2(\Omega)]^d \}, \qquad H_0^1(\Omega) \coloneqq \{ \phi \in H^1(\Omega) \mid \phi = 0 \text{ on } \partial\Omega \}.
    \end{equation*}
    On the interface $\Gamma$, we make use of the trace space $H_{00}^{1/2}(\Gamma)$ from \cite{lions2012non}, consisting of restrictions of $H_0^1(\Omega)$ functions to $\Gamma$.
    Its dual space is denoted by $H^{-1/2}(\Gamma)$, equipped with the norm
    \begin{equation*}
        \lVert \psi \rVert_{H^{-1/2}(\Gamma)} \coloneqq \sup_{\phi \in H^1_0(\Omega)} \frac{\langle \psi, \phi \rangle_{\Gamma}}{\lVert \phi \rVert_{H^1(\Omega)}}.
    \end{equation*}

    For the first result, consider the collective interface term  $\alpha \mathbf{q}_\ell \cdot \mathbf{n}_\ell + g_\ell$. Recall from \eqref{eq:14} that this term corresponds to the interface pressure $p_\ell|_\Gamma$. This post-processing is directly available using the main variables, and we consider its convergence in the following lemma.
    \begin{lemma} \label{lemma4}
     The post-processed interface pressures $p^{i}|_{\Gamma}=\alpha \mathbf{q}^{i} \cdot \mathbf{n} +  g^{i}$ converge strongly to $p^{n}|_{\Gamma}=\alpha \mathbf{q}^{n} \cdot \mathbf{n} +  g^{n}$ in $G$. In particular, 
        \begin{equation*}
            \lim_{i \to \infty} \lVert\alpha e_{\mathbf{q}}^{i} \cdot \mathbf{n} +  e_{g}^{i} \rVert_G = 0.
        \end{equation*}
    \end{lemma}
    \begin{proof}
        Starting from \eqref{eq:22a}
        \begin{equation*}
            \langle (\bar{k}^{n} - \bar{k}^{i-1}) \mathbf{q}^{n}, \tilde{\mathbf{q}} \rangle
             - \langle \bar{k}^{i-1} e_{\mathbf{q}}^i, \tilde{\mathbf{q}} \rangle
            - \langle e_{p}^{i}, \nabla \cdot \tilde{\mathbf{q}} \rangle 
            + \alpha \langle e_{\mathbf{q}}^{i} \cdot \mathbf{n}, \tilde{\mathbf{q}} \cdot \mathbf{n} \rangle_G
            + \langle e_{g}^{i}, \tilde{\mathbf{q}} \cdot \mathbf{n} \rangle_G 
            = 0.
        \end{equation*}
        By rearranging, we get
        \begin{equation*}
            \langle \alpha e_{\mathbf{q}}^{i} \cdot \mathbf{n}
            +  e_{g}^{i}, \tilde{\mathbf{q}} \cdot \mathbf{n} \rangle_G 
            =\langle \bar{k}^{i-1} e_{\mathbf{q}}^i, \tilde{\mathbf{q}} \rangle
            + \langle e_{p}^{i}, \nabla \cdot \tilde{\mathbf{q}} \rangle
            -\langle (\bar{k}^{n} - \bar{k}^{i-1}) \mathbf{q}^{n}, \tilde{\mathbf{q}} \rangle.
        \end{equation*}
        For ease of reading, we define $\gamma^i  \coloneqq  \alpha e_{\mathbf{q}}^{i} \cdot \mathbf{n} +  e_{g}^{i}$. 
        By applying the Cauchy-Schwarz inequality, we get
        \begin{equation*}
            \langle \gamma^i, \tilde{\mathbf{q}} \cdot \mathbf{n} \rangle_G 
            \leq \lVert\bar{k}^{i-1} e_{\mathbf{q}}^i \rVert \lVert \tilde{\mathbf{q}} \rVert
            + \lVert e_{p}^{i} \rVert \lVert \nabla \cdot \tilde{\mathbf{q}} \rVert
            + \lVert (\bar{k}^{n} - \bar{k}^{i-1}) \mathbf{q}^{n} \rVert \lVert \tilde{\mathbf{q}} \rVert.
        \end{equation*}
        Using the Lipschitz continuity and boundedness of $\bar{k}$, and the bound on $\mathbf{q}^n$ from Assumption~\ref{AssumptionTwo}, we have
        \begin{align} \label{eq: interfaceBound}
            \langle \gamma^i, \tilde{\mathbf{q}} \cdot \mathbf{n} \rangle_G
            &\leq M_{\bar{k}}\lVert e_{\mathbf{q}}^{i} \rVert \lVert\tilde{\mathbf{q}} \rVert
            + \lVert e_{p}^{i} \rVert \lVert\nabla \cdot \tilde{\mathbf{q}} \rVert
            + L_{\bar{k}} M_{\mathbf{q}} \lVert e^{i-1}_{S} \rVert
            \lVert \tilde{\mathbf{q}} \rVert \nonumber \\
            &\le ( M_{\bar{k}}\lVert e_{\mathbf{q}}^{i} \rVert
            + \lVert e_{p}^{i} \rVert
            + L_{\bar{k}} M_{\mathbf{q}} \lVert e^{i-1}_{S} \rVert) 
            (\lVert \tilde{\mathbf{q}} \rVert + \lVert \nabla \cdot \tilde{\mathbf{q}} \rVert ).
        \end{align}
        With $\gamma^i \in G \subset [H^{-1/2}(\Gamma)]^2$ and $\tilde{\mathbf{q}} \in Q$, we proceed with a particular choice of $\tilde{\mathbf{q}}$. Let $\mathcal{R}_\ell: H^{-1/2}(\Gamma) \to H(\mathrm{div}; \Omega_\ell)$ be the bounded extension operator from \cite[Section 4.1.2]{Quarteroni1999}, 
        with $(\mathcal{R}_\ell \psi_\ell \cdot \mathbf{n}_\ell)|_\Gamma = \psi_\ell$ on $\Gamma$. 
        For this extension, the following bounds are satisfied
        \begin{equation} \label{eq: boundedness extension}
            \lVert \psi_\ell \rVert_{H^{-1/2}(\Gamma)}
            \lesssim \lVert \mathcal{R}_\ell \psi_\ell \rVert_{H(\mathrm{div};\Omega_\ell)}
            \lesssim \lVert \psi_\ell \rVert_{H^{-1/2}(\Gamma)}, 
            \qquad \forall \ \psi_\ell \in H^{-1/2}(\Gamma).
        \end{equation}  
        Let $\tilde{\mathbf{q}}$ be formed using these extensions of the normal trace, i.e., such that $\tilde{\mathbf{q}}_\ell = \mathcal{R}_\ell \gamma_\ell^i$. The bound \eqref{eq: boundedness extension} and the continuous embedding $L^2(\Gamma) \hookrightarrow H^{-1/2}(\Gamma)$, give us
        \begin{equation} \label{eq: embedding bound}
            \lVert\mathcal{R}_\ell \gamma_\ell^i \rVert_{H(\mathrm{div}; \Omega_\ell)}
            \lesssim \lVert\gamma_\ell^i \rVert_{H^{-1/2}(\Gamma)}
            \lesssim \lVert\gamma_\ell^i \rVert_{L^2(\Gamma)}.
        \end{equation}
        By the definition of the norm $\lVert \cdot \rVert$ from \eqref{Notation1} and the $H(\mathrm{div};\Omega_\ell)$ norm from \eqref{eq:Hdiv norm}, we arrive at a bound for the extension, that is
        \begin{equation} \label{eq: extensionBound}
            \lVert \tilde{\mathbf{q}} \rVert + \lVert \nabla \cdot \tilde{\mathbf{q}} \rVert
            \eqsim \sum_{\ell = 1}^2 
            \lVert\mathcal{R}_\ell \gamma_\ell^i \rVert_{H(\mathrm{div}; \Omega_\ell)}
            \lesssim \lVert\gamma^i \rVert_G.
        \end{equation}
        Substituting the test function in \eqref{eq: interfaceBound} with $\gamma^i$, the bound \eqref{eq: extensionBound} gives
        \begin{equation*}
            \lVert\gamma^i \rVert^2_G 
            \lesssim (M_{\bar{k}}\lVert e_{\mathbf{q}}^{i} \rVert
            + \lVert e_{p}^{i} \rVert
            + L_{\bar{k}} M_{\mathbf{q}} \lVert e^{i-1}_{S} \rVert) \|\gamma^i\|_G.
        \end{equation*} Using the limits from Theorems \ref{thm: convergence S q} and \ref{thm: convergence pressure}, we arrive at
        \begin{equation*}
            \lim_{i \to \infty} \lVert\alpha e_{\mathbf{q}}^{i} \cdot \mathbf{n}
            +  e_{g}^{i} \rVert_G
            = \lim_{i \to \infty} \lVert\gamma^i \rVert_G 
            \lesssim \lim_{i \to \infty} ( M_{\bar{k}}\lVert e_{\mathbf{q}}^{i} \rVert
            + \lVert e_{p}^{i} \rVert
            + L_{\bar{k}} M_{\mathbf{q}} \lVert e^{i-1}_{S} \rVert)
            =0.
        \end{equation*}
    \end{proof}
    Next, we consider the Robin-variable $g$. We already showed weak convergence in $G$ for this variable in Lemma \ref{lem: weak convergence p and g}. The following lemma shows strong convergence in a weaker norm, namely by considering
    \begin{align} \label{eq:equiv norm}
        \Phi_\ell &\coloneqq \{\phi_\ell \in [H^1(\Omega_\ell)]^d : \phi_\ell|_{\partial \Omega} = 0\}, &
        \lVert g \rVert_{G, -1/2} 
        &\coloneqq 
        \sum_{\ell = 1}^2
        \sup_{\phi_\ell \in \Phi_\ell}
        \frac{\langle g_\ell, \phi_\ell \cdot \mathbf{n} \rangle_\Gamma}
        {\lVert \phi_\ell \rVert_{H^1(\Omega_\ell)}}.
    \end{align}

    \begin{theorem} \label{thm: g convergence in H1/2}
        The Robin-variable $g^{i}$ converges strongly to $g^n$ in the norm from \eqref{eq:equiv norm}. In particular,
        \begin{equation*}
            \lim_{i \to \infty} \lVert e_g^i \rVert_{G, -1/2} = 0.
        \end{equation*}
        \end{theorem}
        
        \begin{proof}
        Starting from \eqref{eq:beforesumming}
        \begin{equation*}
            \langle \bar{k}_\ell^{n}\mathbf{q}_\ell^{n} - \bar{k}_\ell^{i-1}\mathbf{q}_\ell^{i}, \tilde{\mathbf{q}}_\ell \rangle_{\Omega_\ell} 
            - \langle e_{p_\ell}^{i}, \nabla \cdot \tilde{\mathbf{q}}_\ell \rangle_{\Omega_\ell} 
            + \alpha \langle e_{\mathbf{q},\ell}^{i} \cdot \mathbf{n}_\ell, \tilde{\mathbf{q}}_\ell \cdot \mathbf{n}_\ell \rangle_\Gamma
            + \langle e_{g,\ell}^{i}, \tilde{\mathbf{q}}_\ell \cdot \mathbf{n}_\ell \rangle_\Gamma
            = 0,
        \end{equation*}
        since $\Phi_\ell \subseteq Q_\ell$, we may choose the test function $\mathbf{\tilde{q}}_\ell$ from that space, i.e., $\phi_\ell$. By considering the numerator in \eqref{eq:equiv norm}, we rearrange the above equation and apply the Cauchy-Schwarz and trace inequalities, thus giving us
        \begin{equation*}
        \begin{aligned}
            \langle e_{g, \ell}, \phi_\ell \cdot \mathbf{n}_\ell \rangle_\Gamma 
            &= -\alpha \langle e_{\mathbf{q}, \ell} \cdot \mathbf{n}_\ell, \phi_\ell \cdot \mathbf{n}_\ell \rangle_\Gamma + \langle e_{p, \ell}, \nabla \cdot \phi_\ell \rangle_{\Omega_\ell} 
            - \langle \bar{k}_\ell^n \mathbf{q}_\ell^{n} - \bar{k}_\ell^{i-1} \mathbf{q}_\ell^{i}, \phi_\ell \rangle_{\Omega_\ell} \\
            &\leq \alpha \lVert e_{\mathbf{q}, \ell} \cdot \mathbf{n}_\ell \rVert_{H^{-1/2}(\Gamma)} \lVert \phi_\ell \cdot \mathbf{n}_\ell \rVert_{H^{1/2}(\Gamma)}
            + \lVert e_{p, \ell} \rVert_{\Omega_\ell} \lVert \nabla \cdot \phi_\ell \rVert_{\Omega_\ell} 
            + \lVert\bar{k}_\ell^n \mathbf{q}_\ell^{n} - \bar{k}_\ell^{i-1} \mathbf{q}_\ell^{i} \rVert_{\Omega_\ell}\lVert \phi_\ell \rVert_{\Omega_\ell}\\ 
            &\lesssim \alpha \lVert e_{\mathbf{q}, \ell} \rVert_{H(\mathrm{div};\Omega_\ell)} \lVert \phi_\ell \rVert_{H^{1}(\Omega_\ell)} 
            + (\lVert e_{p, \ell} \rVert_{\Omega_\ell}+ \lVert\bar{k}_\ell^n \mathbf{q}_\ell^{n} - \bar{k}_\ell^{i-1} \mathbf{q}_\ell^{i} \rVert_{\Omega_\ell}) \lVert\phi_\ell \rVert_{H^{1}(\Omega_\ell)}.
        \end{aligned}
        \end{equation*}
        Dividing by $\lVert \phi_\ell \rVert_{H^1(\Omega_\ell)}$, taking the supremum over $\phi_\ell \in \Phi_\ell$, and adding over the subdomains, we obtain
        \begin{equation*}
            \lVert e_g^i \rVert_{G, -1/2}
            \lesssim 
            \alpha (\lVert e_{\mathbf{q}}^i \rVert + \lVert \nabla \cdot e_{\mathbf{q}}^i \rVert)
            + \lVert e_p^i \rVert
            + \lVert \bar{k}^n \mathbf{q}^n - \bar{k}^{i-1} \mathbf{q}^i \rVert.
        \end{equation*}
        The last term is handled, as before, by adding and subtracting $\bar{k}^{i-1}\mathbf{q}^{n}$. Then Theorem~\ref{thm: convergence S q}, Theorem~\ref{thm: convergence pressure}, and Theorem~\ref{theorem3} imply that all terms on the right-hand side converge to zero as $i \to \infty$, yielding the desired result.
        \end{proof}

        As the last result, we consider the normal flux $\mathbf{q} \cdot \mathbf{n}$, and present two corollaries showing strong convergence in the weaker norm \Cref{eq:equiv norm}.
        \begin{corollary} \label{thm:weak q n convergence}
            The normal flux $\mathbf{q}^{i} \cdot \mathbf{n}$ converges to $\mathbf{q}^n \cdot \mathbf{n}$ in the sense that
            \begin{equation*}
                \lim_{i \to \infty} \lVert e_\mathbf{q}^i \cdot \mathbf{n} \rVert_{G, -1/2} = 0.
            \end{equation*}
        \end{corollary}
        \begin{proof}
        By adding and subtracting $e_{g}^{i}$, and applying the triangle inequality, we get
        \begin{equation*}
            \alpha \lVert e_{\mathbf{q}}^{i} \cdot \mathbf{n} \rVert_{G, -1/2} 
            \leq \lVert\alpha e_{\mathbf{q}}^{i} \cdot \mathbf{n} + e_{g}^{i} \rVert_{G, -1/2} 
            + \lVert e_{g}^{i} \rVert_{G, -1/2}.
        \end{equation*}
        Taking the limit as $i \to \infty$ and using Lemma \ref{lemma4} and Lemma \ref{thm: g convergence in H1/2}, remembering that $\alpha > 0$, we conclude the result.
    \end{proof}

     \begin{corollary} \label{lem:flux continuity}
            In the limit, the solution has normal flux continuity across the interface $\Gamma$ in the sense that
            \begin{align}
                \lim_{i \to \infty} \lVert \mathbf{q}^{i}_1 \cdot \mathbf{n}_1 + \mathbf{q}^{i}_2 \cdot \mathbf{n}_2 \rVert_{H^{-1/2}(\Gamma)} 
            = 0.
            \end{align}
        \end{corollary}
    \begin{proof}
        Using the fact that $\mathbf{q}^{n}_1 \cdot \mathbf{n}_1 + \mathbf{q}^{n}_2 \cdot \mathbf{n}_2 = 0$, we apply a triangle inequality to derive
        \begin{align}
            \lVert \mathbf{q}^{i}_1 \cdot \mathbf{n}_1 + \mathbf{q}^{i}_2 \cdot \mathbf{n}_2 \rVert_{H^{-1/2}(\Gamma)} 
            = \lVert e_{\mathbf{q}, 1}^{i} \cdot \mathbf{n}_1 + e_{\mathbf{q}, 2}^{i} \cdot \mathbf{n}_2 \rVert_{H^{-1/2}(\Gamma)} 
            \lesssim \lVert e_{\mathbf{q}}^{i} \cdot \mathbf{n} \rVert_{G, -1/2}.
        \end{align}
        The result follows from Corollary~\ref{thm:weak q n convergence}.
    \end{proof}
    This concludes our analysis section. 

    \section{Numerical results} \label{sec:numeric}
    In this section, we explore three numerical experiments. We begin with two 2D problems that admit a manufactured solution, followed by a 3D problem representing a more realistic case using van Genuchten parameterization. Since the mLRDD-scheme presented in \eqref{eq:final scheme} is not tied to any specific spatial discretization, we can employ any mixed-form discretization, as long as the discrete spaces satisfy $Q_{\ell,h} \subset Q_\ell$ and $P_{\ell,h} \subset P_\ell$, as defined in \eqref{eq:local flux space}. In particular, the normal traces of the discrete flux space must be in $L^2(\Gamma)$. To this end, we employ a standard mixed finite element method. More precisely, we use piecewise constant ($\mathbb{P}_0$) and lowest-order Raviart-Thomas ($\mathbb{RT}_0$) elements \cite{raviart2006mixed} to approximate the pressure and flux, respectively (see \cite{brezzi2012mixed}). This element pair is locally mass conservative and is supported by \textit{a priori} error estimates \cite{raviart2006mixed, brezzi2012mixed}, providing a theoretical basis for the numerical results. All domains $\Omega$ considered in this section are rectangular and are partitioned into unstructured, conforming grids composed of closed \textit{d}-simplices.

    Throughout this section, the mLRDD-scheme is compared to other well-established linearization methods for Richards' equation, namely the Newton method \cite{bergamaschi1999mixed}, modified Picard \cite{celia1990general}, and L-scheme \cite{slodicka2002robust, pop2004mixed, list2016study}. These methods are implemented using a monolithic approach, that is, each scheme is defined on the full domain $\Omega$, where the interface $\Gamma$ only acts as a restriction for the grid assembly. To this extent, we introduce the global flux space for the monolithic schemes as
    \begin{equation*}
        \tilde{Q}  \coloneqq \{
            \tilde{\mathbf{q}} \in H(\mathrm{div}; \Omega) 
            \mid (\tilde{\mathbf{q}} \cdot \mathbf{n})|_{\partial \Omega_N} =0 \}.
    \end{equation*}
    The first presented monolithic scheme is the quadratically convergent monolithic Newton scheme:
    
    \begin{problem}[Monolithic Newton] \label{Newton Method}
        Given $\mathbf{q}^{n,i-1} \in \tilde{Q}$ and $p^{n,i-1} \in P$, find $\mathbf{q}^{n,i}\in \tilde{Q}$ and $p^{n,i}\in P$ such that
        \begin{alignat*}{2}
            \langle \bar{k}^{n,i-1}\mathbf{q}^{n,i}, \tilde{\mathbf{q}} \rangle 
            + \langle \bar{k}^{{n,i-1}^\prime}\mathbf{q}^{n,i-1}p^{n,i}, \tilde{\mathbf{q}} \rangle
            - \langle p^{n,i}, \nabla \cdot \tilde{\mathbf{q}} \rangle
            &= \langle \nabla \eta, \tilde{\mathbf{q}} \rangle
            - \langle p^{n}, \tilde{\mathbf{q}} \cdot \mathbf{n} 
            \rangle_{\partial\Omega} 
            + \langle \bar{k}^{{n,i-1}^\prime}\mathbf{q}^{n,i-1}p^{n,i-1}, \tilde{\mathbf{q}} \rangle, \\
            \langle S^{{n,i-1}^\prime} p^{n,i}, \tilde{p} \rangle
            + \tau \langle \nabla \cdot \mathbf{q}^{n,i}, \tilde{p} \rangle
            &= \tau\langle \mathcal{F}, \tilde{p} \rangle
            - \langle S^{n,i-1} - S^{n-1}, \tilde{p} \rangle
            + \langle S^{{n,i-1}^\prime} p^{n,i-1}, \tilde{p} \rangle,
        \end{alignat*}
        for all $\tilde{\mathbf{q}} \in \tilde{Q}$ and $\tilde{p} \in P$.
    \end{problem}
    \noindent
    Next, we present the monolithic modified Picard scheme, hereafter referred to as monolithic Picard, given as: 
    \begin{problem}[Monolithic Picard]
        Given $\mathbf{q}^{n,i-1} \in \tilde{Q}$ and $p^{n,i-1} \in P$, find $\mathbf{q}^{n,i}\in \tilde{Q}$ and $p^{n,i}\in P$ such that
        \begin{alignat*}{2}
            \langle \bar{k}^{n,i-1}\mathbf{q}^{n,i}, \tilde{\mathbf{q}} \rangle 
            - \langle p^{n,i}, \nabla \cdot \tilde{\mathbf{q}} \rangle
            &= \langle \nabla \eta, \tilde{\mathbf{q}} \rangle
            - \langle p^{n}, \tilde{\mathbf{q}} \cdot \mathbf{n} 
            \rangle_{\partial\Omega}, \\
            \langle S^{{n,i-1}^\prime} p^{n,i}, \tilde{p} \rangle
            + \tau \langle \nabla \cdot \mathbf{q}^{n,i}, \tilde{p} \rangle
            &= \tau\langle \mathcal{F}, \tilde{p} \rangle
            - \langle S^{n,i-1} - S^{n-1}, \tilde{p} \rangle
            + \langle S^{{n,i-1}^\prime} p^{n,i-1}, \tilde{p} \rangle,
        \end{alignat*}
        for all $\tilde{\mathbf{q}} \in \tilde{Q}$ and $\tilde{p} \in P$.
    \end{problem}
    \noindent
    Finally, we compare the mLRDD-scheme to its monolithic counterpart, the L-scheme:
    \begin{problem}[Monolithic L-scheme]
        Given $\mathbf{q}^{n,i-1} \in \tilde{Q}$ and $p^{n,i-1} \in P$, find $\mathbf{q}^{n,i}\in \tilde{Q}$ and $p^{n,i}\in P$ such that
        \begin{alignat*}{2}
            \langle \bar{k}^{n,i-1}\mathbf{q}^{n,i}, \tilde{\mathbf{q}} \rangle 
            - \langle p^{n,i}, \nabla \cdot \tilde{\mathbf{q}} \rangle
            &= \langle \nabla \eta, \tilde{\mathbf{q}} \rangle
            - \langle p^{n}, \tilde{\mathbf{q}} \cdot \mathbf{n} \rangle_{\partial\Omega}, \\
            L\langle p^{n,i}, \tilde{p} \rangle 
            + \tau \langle \nabla \cdot \mathbf{q}^{n,i}, \tilde{p} \rangle
            &= \tau\langle \mathcal{F}, \tilde{p} \rangle
            - \langle S^{n,i-1} - S^{n-1}, \tilde{p} \rangle
            + L\langle p^{n,i-1}, \tilde{p} \rangle,
        \end{alignat*}
        for all $\tilde{\mathbf{q}} \in \tilde{Q}$ and $\tilde{p} \in P$.
    \end{problem}

    To monitor convergence of the solver, a residual-based stopping criterion is used for both the mass balance \eqref{eq:1b} and the Darcy flow equation \eqref{eq:1a}. The residuals $r_\mathbf{q}$ and $r_p$ are computed by substituting the current iterates $\mathbf{q}^{n,i}, p^{n,i}$ into the nonlinear saddle-point system \eqref{eq:limit solution} and subtracting the right-hand sides $b_\mathbf{q}, b_p$, with nonlinear coefficients evaluated at the current iterate. The iteration is deemed converged when either the absolute residual or the residual relative to the right-hand side falls below $10^{-6}$. That is,
    \begin{equation}\label{eq:residual}
        \max_{u \in \{p, \mathbf{q}\}}\min\left(\lVert r_u \rVert_2, \lVert r_u \rVert_2^{rel} \right) < 10^{-6}, \qquad \text{with} \quad \lVert r_u \rVert_2^{rel} \coloneqq \frac{\lVert r_u \rVert_2}{\lVert b_u \rVert_2 + 10^{-14}},
    \end{equation}
    where $\lVert \cdot \rVert_2$ is the Euclidean norm. A small regularization term is added to the relative residual denominator to prevent division by zero. Physically, evaluating these residuals provides a direct measure of the local mass conservation error and the deviation from Darcy's law, ensuring that the converged numerical scheme strictly preserves physical mass and flow dynamics.

    The numerical simulations were run using the Windows Subsystem for Linux (WSL) on a 64-bit Windows 11 host equipped with a 13th Gen Intel Core i7-1370P processor (1.90 GHz) and 32 GB of RAM. The numerical framework was implemented using PyGeon, a Python package for mixed-dimensional discretization in porous media \cite{pygeon}, which expands the PorePy simulation tool \cite{keilegavlen2021porepy}.

    \subsection{Verification via Manufactured Solution}
    \label{subsection: MMS}
    To rigorously verify the spatial and temporal convergence of the scheme, as well as its ability to handle spatially discontinuous material properties, we apply the Method of Manufactured Solutions (MMS).
    
    \subsubsection{Example 1: 2D problem with manufactured solution allowing varying subdomain permeability} \label{example 1}
    
    We consider a two-dimensional, unsaturated porous medium split into two distinct subdomains. The domains $\Omega_\ell$ and the interface $\Gamma$ are defined as:
    \begin{equation*}
        \Omega_1 = (-1,0) \times (0,1), \qquad \Omega_2 = (0,1) \times (0,1), \qquad \Gamma = \{0\} \times [0,1].
    \end{equation*}
    The interface $\Gamma$ represents a sharp discontinuity in the absolute permeability. Neglecting gravity, we prescribe a transient, continuous manufactured pressure field that is strictly negative for $t > 0$, ensuring the global domain remains unsaturated. We first define the spatial functions,  
    \begin{equation*}
        \Psi(x,y) = 10 + x + x(3y^2 - 2y^3) \qquad \text{and} \qquad
        \Theta_\ell(x,y) = \left[ 10^{K_\ell-1} \Psi(x,y) \right]^{1/K_\ell}, \qquad \ell \in \{1,2\}. 
    \end{equation*}
    The domain-dependent pressure field is then given by
    \begin{equation*}
        p(x,y,t) = 1 - (1+t^2)\Theta_\ell(x,y), \qquad (x,y) \in \Omega_\ell, \qquad t > 0,
    \end{equation*}
    for $\ell \in \{1,2\}$. While the saturation-pressure relationship $S(p)$ is uniform across the domain, the effective permeability $k_{\text{eff},\ell}(S)$ introduces heterogeneity across the interface. These are defined as
    \begin{equation*}
        S(p) = \begin{cases} 
            (1 - p)^{-1/2}, & p < 0 \\
            1, & p \geq 0 
        \end{cases} 
        \qquad \text{and} \qquad 
        k_{\text{eff},\ell}(S) = K_\ell S(p)^{2}, \qquad \ell \in \{1,2\},
    \end{equation*}
    where $K_\ell$ represents the absolute permeability of subdomain $\Omega_\ell$. The absolute permeability is set to $K_1=1, K_2=2$ throughout \hyperref[example 1]{Example 1}, except when specified in the parameter study (see Section \ref{par:parameter study}).
    
    When applying these functions to the flow equation \eqref{eq:2b}, this specific choice of effective permeability exactly cancels the time-dependent and material-dependent components of the pressure gradient. This yields a stationary (time-independent) manufactured flux solution that is globally continuous,
    \begin{equation*}
        \mathbf{q}(x,y) = 
            \Psi(x,y)^{-1} \begin{pmatrix} 1 + 3y^2 - 2y^3 \\[5pt] 6xy - 6xy^2 \end{pmatrix}, \qquad (x,y) \in \Omega_1 \cup \Omega_2.
    \end{equation*}
    Finally, substituting the flux and saturation into the governing mass balance equation \eqref{eq:2a} provides the necessary right-hand side source term to drive the system:
    \begin{equation*}
        f(x,y,t) = -\frac{t}{\sqrt{(1+t^2)^3 \Theta_\ell}} 
        + \frac{(6x-12xy)\Psi - \left[(1+3y^2-2y^3)^2+(6xy-6xy^2)^2 \right]}{\Psi^2}, \quad (x,y) \in \Omega_\ell, \quad t > 0,
    \end{equation*}
    for $\ell \in \{1,2\}$.
    
     \begin{table}[pos=ht]
        \centering
        \caption {Domain-dependent initial and boundary conditions for \hyperref[example 1]{Example 1}.}
        \usebox{\CondExOne}
    \end{table}
        
    \subsubsection{Example 2: 2D problem with manufactured solution adapted from literature} \label{example 2}
    We again consider a two-dimensional, unsaturated porous medium split into two distinct subdomains. For this example, we adapt the manufactured solution from \cite[Section 4.1]{seus2018linear}. The domains $\Omega_\ell$ and the internal interface $\Gamma$ are defined as:
    \begin{equation*}
        \Omega_1 = (-1,0) \times (0,1), \qquad \Omega_2 = (0,1) \times (0,1), \qquad \Gamma = \{0\} \times [0,1].
    \end{equation*}
    Neglecting gravity, we prescribe a manufactured pressure field that is strictly negative for $t > 0$, ensuring the domain remains unsaturated. The domain-dependent pressure field is then given by
    \begin{equation*}
        \begin{aligned}
            p(x,y,t)  = \begin{cases} 
                1 - (1+t^2)(1+x^2+y^2), & \qquad (x,y) \in \Omega_1 \\ 
                1 - (1+t^2)(1+y^2), & \qquad (x,y) \in \Omega_2
            \end{cases}, \qquad t > 0, 
        \end{aligned}
    \end{equation*}
    for $\ell \in \{1,2\}$. The saturation $S_\ell(p)$ and the permeability $k_{\ell}(S_\ell(p))$ are defined as
    \begin{equation*}
        S_\ell(p) = \begin{cases} 
            (1 - p)^{-\frac{1}{1+\ell}}, & p < 0 \\
            1, & p \geq 0 
        \end{cases} 
        \qquad \text{and} \qquad 
        k_{\ell}(S_\ell) = S_\ell(p)^{1+\ell}, \qquad \ell \in \{1,2\}.
    \end{equation*}    
    When applying these functions to the flow equation \eqref{eq:2b}, we again get a stationary (time-independent) manufactured flux solution, continuous across the interface
    \begin{equation} \label{eq:flux ex 2}
        \mathbf{q}_1(x,y) = 
            \begin{pmatrix} \frac{2x}{1+x^2+y^2}\\ \frac{2y}{1+x^2+y^2} \end{pmatrix}, \qquad (x,y) \in \Omega_1, \qquad \mathbf{q}_2(x,y) = 
            \begin{pmatrix} 0 \\ \frac{2y}{1+y^2} \end{pmatrix}, \qquad (x,y) \in \Omega_2.
    \end{equation}
    Finally, substituting the flux and saturation into the governing mass balance equation \eqref{eq:2a} provides the necessary right-hand side source term to drive the system:
    \begin{equation*}
        f(x,y,t) = \begin{cases}  \frac{4}{(1+x^2+y^2)^2} 
        - \frac{t}{\sqrt{(1+t^2)^3(1+x^2+y^2)}}, & \qquad (x,y) \in \Omega_1 \\ 
        \frac{2(1-y^2)}{(1+y^2)^2} 
        - \frac{2t}{3\sqrt[3]{(1+t^2)^4(1+y^2)}}, & \qquad (x,y) \in \Omega_2
        \end{cases}, \qquad t > 0,
    \end{equation*}
    for $\ell \in \{1,2\}$. 

    \begin{table}[pos=ht]
        \centering
        \footnotesize
        \caption{Domain-dependent initial and boundary conditions for \hyperref[example 2]{Example 2}.}
        \usebox{\CondExTwo}
    \end{table}
            
    \begin{rmk}
    To have a correct formulation for the manufactured solutions in \ref{subsection: MMS}, the solutions must implicitly preserve the physical transmission conditions across the interface, i.e., $\llbracket \mathbf{q} \cdot \mathbf{n}\rrbracket= 0 $ and $\llbracket p \rrbracket= 0$. By evaluating $p(x=0,y,t)$ and $\mathbf{q}(x=0,y) \cdot \mathbf{n}$, it is trivial to show that the pressure is continuous across the interface $\Gamma$.
    \end{rmk}

    \subsubsection{Results}
    The following section collects results from \hyperref[example 1]{Example 1} and \hyperref[example 2]{Example 2}. The stabilization parameter $L$ and Robin-parameter $\alpha$ for \hyperref[example 1]{Example 1} and \hyperref[example 2]{Example 2} are chosen to be $L = 0.013, \alpha=0.6$ and $L=0.26,\alpha=0.21$, respectively, except when explicitly stated. As unstructured grids are employed, the mean cell diameter $h$ over the global domain in question serves as the mesh size.
    
    \paragraph{Solution field.} Figures~\ref{fig:pressure map ex 1} and \ref{fig:pressure map ex 2} show the pressure map $p(x,y)$ with flux field for \hyperref[example 1]{Example 1} and \hyperref[example 2]{Example 2}, respectively, at the final time $T=0.5$, with timestep $\tau=0.01$. For \hyperref[example 2]{Example 2}, one can clearly see in Figure~\ref{fig:pressure map ex 2} that the flux field is parallel to the interface at $x=0$, yielding zero normal flux at the interface, matching the manufactured flux solution in \eqref{eq:flux ex 2}. From the two solution fields, one can also observe that the pressure continuity at the interface is achieved.

    \begin{figure}
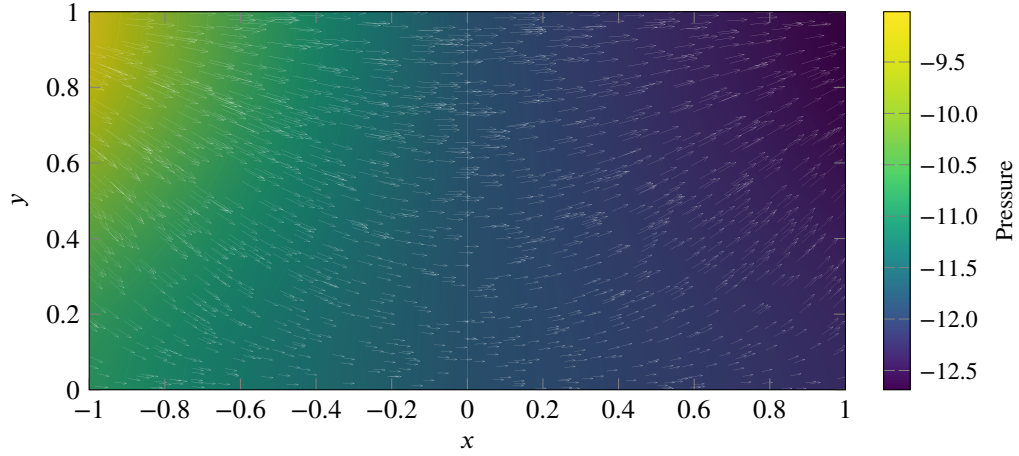

        \centering
        \usebox{\simulationExOne}
        \caption{Pressure profile $p(x,y)$ with flux field $\mathbf{q}$ for \hyperref[example 1]{Example 1} at $T=0.5$, with $\tau=0.01$, $h\approx0.013$.}
        \label{fig:pressure map ex 1}
    \end{figure}

    \begin{figure}
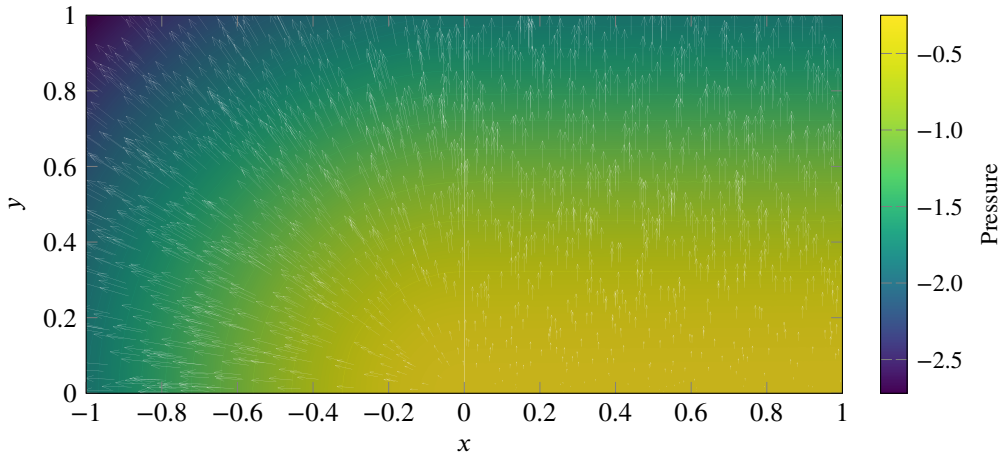

        \centering
        \usebox{\simulationExTwo}
        \caption{Pressure profile $p(x,y)$ with flux field $\mathbf{q}$ for \hyperref[example 2]{Example 2} at $T=0.5$, with $\tau=0.01$, $h\approx0.013$.}
        \label{fig:pressure map ex 2}
    \end{figure}

    \paragraph{Order of accuracy.}
    \label{par: order of accuracy MMS}
    For any variable $u$, let $\varepsilon_u^n = u^n - u_h^{n}$ denote the error between the manufactured solution and its discrete approximation $u_h^{n}$ at time step $n$. Tables~\ref{tab:convergence ex 1} and~\ref{tab:convergence ex 2} report the errors and convergence orders for \hyperref[example 1]{Example 1} and \hyperref[example 2]{Example 2}, respectively, while Tables~\ref{tab:convergence Q ex 1} and \ref{tab:convergence Q ex 2} report the error and convergence order for the components of the $Q$-norm as defined in \eqref{eq:local flux space}. Relative norms $\lVert \varepsilon_u^n \rVert_X / \lVert u^n\rVert_X$ are used for \hyperref[example 1]{Example 1}, while absolute norms $\lVert\varepsilon_u^n\rVert_X$ are used for \hyperref[example 2]{Example 2}, where $X$ denotes the appropriate norm for each variable. This distinction is necessary, as previously mentioned, the manufactured normal flux vanishes identically at the interface (see \ref{eq:flux ex 2}), rendering the relative norm undefined. The temporal discretization error is isolated by setting the timestep size sufficiently small, here chosen as $\tau = 10^{-6}$, with a final time $T = 10^{-5}$. The convergence order at refinement level $j$ for solutions at timestep $n$ is computed as
    \begin{equation*}
        \frac{\log\left(\lVert\varepsilon_{u,j-1}^n\rVert_X \,/\, 
        \lVert\varepsilon_{u,j}^n\rVert_X\right)}
        {\log\left(h_{j-1} \,/\, h_{j}\right)}, \qquad u \in\{p, \mathbf{q}, g\},
    \end{equation*}
    in the appropriate norm $X$. As Tables~\ref{tab:convergence ex 1}, \ref{tab:convergence Q ex 1}, \ref{tab:convergence ex 2}, and~\ref{tab:convergence Q ex 2} show, both examples exhibit first-order convergence in all variables, consistent with the expected order of accuracy for the $\mathbb{RT}_0$--$\mathbb{P}_0$ discretization pair~\cite{brezzi2012mixed, raviart2006mixed}.

    \begin{table}
        \centering
        \caption{Spatial convergence for \hyperref[example 1]{Example 1}: errors $\varepsilon_u$ and convergence orders for pressure $p$, flux $\mathbf{q}$, and Robin-variable $g$ in their respective norms, under mesh refinement with mean cell diameter $h$, evaluated at $T = 10^{-5}$ with $\tau = 10^{-6}$.}
        \label{tab:convergence ex 1}
        \begin{threeparttable}
        \convergenceExOneFirst
        \begin{tablenotes}
            \footnotesize
            \item[$rel:$] Relative norm $\lVert \varepsilon_u^n\rVert_X/ \lVert u^n\rVert_X$.
        \end{tablenotes}
        \end{threeparttable}
    \end{table}
    
    \begin{table}
        \centering
        \caption{Spatial convergence of the decomposed $Q$-norm for \hyperref[example 1]{Example 1}: errors $\varepsilon_u$ and convergence orders for flux $\mathbf{q}$, divergence $\nabla \cdot \mathbf{q}$, and normal trace $\mathbf{q} \cdot \mathbf{n}$ in their respective norms, under mesh refinement with mean cell diameter $h$, evaluated at $T = 10^{-5}$ with $\tau = 10^{-6}$.}
        \label{tab:convergence Q ex 1}
        \begin{threeparttable}
        \convergenceExOneSecond
        \begin{tablenotes}
            \footnotesize
            \item[$rel:$] Relative norm $\lVert \varepsilon_u \rVert_X / \lVert u\rVert_X$.
        \end{tablenotes}
        \end{threeparttable}
    \end{table}

    \begin{table}
        \centering
        \caption{Spatial convergence for \hyperref[example 2]{Example 2}: errors $\varepsilon_u$ and convergence orders for pressure $p$, flux $\mathbf{q}$, and Robin-variable $g$ in their respective norms, under mesh refinement with mean cell diameter $h$, evaluated at $T = 10^{-5}$ with $\tau = 10^{-6}$.}
        \label{tab:convergence ex 2}
        {\convergenceExTwoFirst}
    \end{table}
    
    \begin{table}
        \centering
        \caption{Spatial convergence of the decomposed $Q$-norm for \hyperref[example 2]{Example 2}: errors $\varepsilon_u$ and convergence orders for flux $\mathbf{q}$, divergence $\nabla \cdot \mathbf{q}$ and normal trace $\mathbf{q} \cdot \mathbf{n}$ in their respective norms, under mesh refinement with mean cell diameter $h$, evaluated at $T = 10^{-5}$ with $\tau = 10^{-6}$.}
        \label{tab:convergence Q ex 2}
        {\convergenceExTwoSecond}
    \end{table}
    
    \paragraph{Convergence of the iterative scheme.} \label{par:iterative convergence MMS}
    To evaluate the convergence of the mLRDD-scheme~\ref{FinalScheme} numerically, we monitor the mass and flux residuals \eqref{eq:residual} against the iteration number, as shown in Figure~\ref{fig:inner convergence shared}. The Robin-variable is also monitored, but since $g_\ell^{n,i}$ is set explicitly from the flux and Robin-variable of the neighboring subdomain at the previous iteration (see~\eqref{eq:final scheme c}), no equation residual is defined for $g$. Instead, the iterate difference $\lVert g_\ell^{n,i} - g_\ell^{n,i-1}\rVert_{L^2(\Gamma)}$ is used as a measure of convergence for the interface coupling across $\Gamma$, which vanishes as the transmission conditions are satisfied. The flux jump $\lVert \llbracket \mathbf{q}^{n,i} \cdot \mathbf{n} \rrbracket \rVert_{L^2(\Gamma)}$ is also monitored to verify that flux continuity across $\Gamma$ is recovered upon convergence. For both examples, all monitored quantities exhibit linear convergence. Also, the Robin-variable and flux jump tend to zero, indicating continuity across the interface. The convergence analysis from \ref{sec:analysis} guarantees convergence of the scheme without prescribing a rate; these numerical results supplement the analysis findings by indicating that the guaranteed convergence is first-order. As the tolerance for reaching convergence is governed by both the mass and flow residuals falling below a certain threshold, one can observe in Figure~\ref{fig:inner convergence shared} that the mass residual is the limiting factor for reaching convergence at an earlier iteration.
        
    \begin{figure}
        \centering
        \begin{subfigure}[t]{0.49\textwidth}
            \centering
            \resizebox{0.90\linewidth}{!}{\usebox{\iterErrorExOne}}
            \label{fig:inner convergence ex 1}
        \end{subfigure}
        \hfill
        \begin{subfigure}[t]{0.49\textwidth}
            \centering
            \resizebox{0.90\linewidth}{!}{\usebox{\iterErrorExTwo}}
            \label{fig:inner convergence ex 2}
        \end{subfigure}
    
        \vspace{1ex}
        \centering
        \pgfplotslegendfromname{iter-error-legend}
    
        \vspace{1ex}
        \caption{Convergence of the iterative scheme for \hyperref[example 1]{Example 1} and \hyperref[example 2]{Example 2}: semi-logarithmic plot of the mass residual $\lVert r_p^{n,i} \rVert_2^{rel}$ and flow residual  $\lVert r_\mathbf{q}^{n,i}\rVert_2^{rel}$, the interface flux jump $\lVert \llbracket \mathbf{q}^{n,i} \cdot \mathbf{n} \rrbracket \rVert_{L^2(\Gamma)}$, and the Robin-variable iterate difference $\lVert g^{n,i} - g^{n,i-1}\rVert_{L^2(\Gamma)}^{rel}$, against iteration number $i$, at $T = 0.5$ with $\tau = 0.01$, $h \approx 0.013$.}
        \label{fig:inner convergence shared}
    \end{figure}
    
    To compare the residual decay of the mLRDD-scheme against monolithic solvers and explore mesh-dependencies, we plot the mass and flow residuals against iteration number for varying mesh sizes in Figure~\ref{fig:residuals shared}. As expected for monolithic approaches with a good initial guess, Newton requires the fewest iterations, and is followed closely by Picard. The monolithic L-scheme requires more iterations and varies a lot between the two examples. The proposed mLRDD-scheme requires the most amount of iterations of all the schemes, but is comparable to the monolithic L-scheme in \hyperref[example 2]{Example 2}. Regarding the effect of spatial refinement, the mass residual decay is independent of the mesh size for all solvers, consistent with the theoretical mesh-independence of the L-scheme linearization \cite{list2016study}. A slight mesh dependence is observed in the flow residual across all methods, present equally in the mLRDD-scheme and its monolithic counterparts.
    
    \begin{figure}
        \centering
        \begin{subfigure}[t]{0.49\textwidth}
            \centering
            \resizebox{0.90\linewidth}{!}{\usebox{\residualPComparisonExOne}}
            \label{fig:p_residual_ex_one}
        \end{subfigure}
        \hfill
        \begin{subfigure}[t]{0.49\textwidth}
            \centering
            \resizebox{0.90\linewidth}{!}{\usebox{\residualQComparisonExOne}}
            \label{fig:q_residual_ex_one}
        \end{subfigure}
    
        \vspace{2ex}
    
        \begin{subfigure}[t]{0.49\textwidth}
            \centering
            \resizebox{0.90\linewidth}{!}{\usebox{\residualPComparisonExTwo}}
            \label{fig:p_residual_ex_two}
        \end{subfigure}
        \hfill
        \begin{subfigure}[t]{0.49\textwidth}
            \centering
            \resizebox{0.90\linewidth}{!}{\usebox{\residualQComparisonExTwo}}
            \label{fig:q_residual_ex_two}
        \end{subfigure}
    
        \vspace{2ex}
        \centering
        \pgfplotslegendfromname{shared-residual-legend-ex-two}
    
        \vspace{1ex}
        \caption{Comparison of residual decay between methods for \hyperref[example 1]{Example 1} and \hyperref[example 2]{Example 2}: Semi-logarithmic plot of mass residual $\lVert r_p^{n,i}\rVert_2^{rel}$ and flow residual $\lVert r_\mathbf{q}^{n,i}\rVert_2^{rel}$ decay for the mLRDD-scheme and monolithic Newton, Picard, and L-scheme over different mesh sizes at $T=0.5$, with $\tau = 0.01$.}
        \label{fig:residuals shared}
    \end{figure}

    \paragraph{Robustness.} \label{par: robustness MMS}
    For the next numerical result, we investigate the robustness of the mLRDD-scheme, in comparison to the monolithic methods. The convergence of Newton and Picard is only guaranteed when the initial guess is sufficiently close to the exact solution, which in practice requires sufficiently small $\tau$ when using the solution from the previous time step as initialization \cite{park1995mixed, radu2006newton, list2016study}. The L-scheme is free of this constraint. While the theoretical convergence of the mLRDD-scheme requires the parameter conditions stated in Section~\ref{thm: convergence S q}, including a uniform bound on $\lVert \mathbf{q}^{n,i}_\ell\rVert_{L^\infty}$, which may not be fulfilled under a poor initial guess, we investigate its empirical behavior in this setting.
    
    To this end, we replace the standard initialization $(\mathbf{q}^{n,0}, p^{n,0}) = (\mathbf{q}^{n-1}, p^{n-1})$ with constant initial guesses: $p^{n,0} = 0$ and $\mathbf{q}^{n,0} = (0,0)^T$ for \hyperref[example 1]{Example 1} and $p^{n,0} = -3$ and $\mathbf{q}^{n,0} = (1,1)^T$ for \hyperref[example 2]{Example 2}. As shown in Figure~\ref{fig:wrong guess}, Newton fails to converge in both examples, exhibiting persistent oscillations, terminated after 100 iterations; only the first 50 of 100 iterations are shown for clarity. Picard converges in \hyperref[example 1]{Example 1} after an initial oscillatory phase, but diverges in \hyperref[example 2]{Example 2}. Both the L-scheme and the mLRDD-scheme exhibit monotone residual decay from the first iteration, demonstrating robust convergence. Again, the mLRDD-scheme uses the most amount of iterations, but is comparable to the L-scheme for \hyperref[example 2]{Example 2}.
    
    \begin{figure}
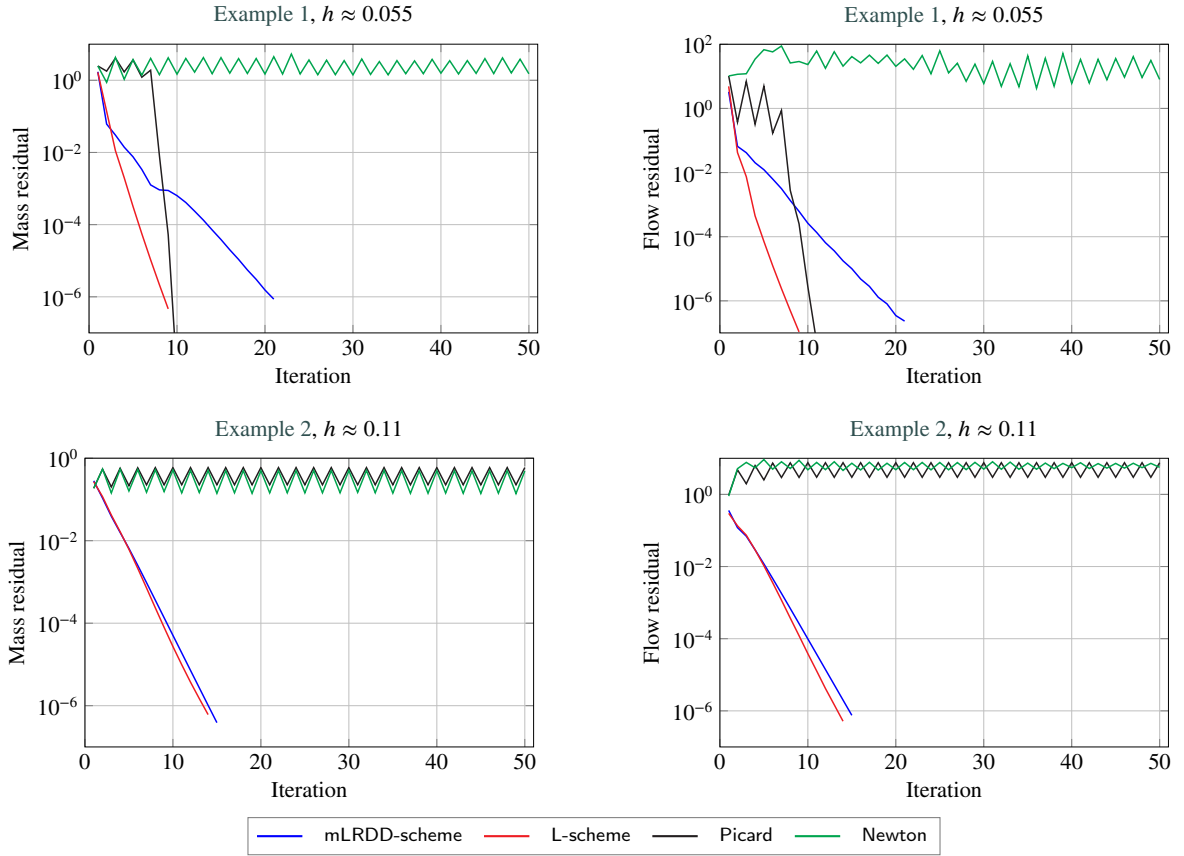

        \centering
        \begin{subfigure}[t]{0.49\textwidth}
            \centering
            \resizebox{0.90\linewidth}{!}{\usebox{\wrongGuessPExOne}}
        \end{subfigure}
        \hfill
        \begin{subfigure}[t]{0.49\textwidth}
            \centering
            \resizebox{0.90\linewidth}{!}{\usebox{\wrongGuessQExOne}}
            \label{fig:wrong guess q ex 1}
        \end{subfigure}
    
        \vspace{2ex}
    
        \begin{subfigure}[t]{0.49\textwidth}
            \centering
            \resizebox{0.90\linewidth}{!}{\usebox{\wrongGuessPExTwo}}
            \label{fig:wrong guess p ex 2}
        \end{subfigure}
        \hfill
        \begin{subfigure}[t]{0.49\textwidth}
            \centering
            \resizebox{0.90\linewidth}{!}{\usebox{\wrongGuessQExTwo}}
            \label{fig:wrong guess q ex 2}
        \end{subfigure}
        
        \vspace{1ex}
        \centering
        \pgfplotslegendfromname{wrong-guess-error-legend}
    
        \caption{Residual decay for different schemes with wrong initial guess for \hyperref[example 1]{Example 1} and \hyperref[example 2]{Example 2}: Semi-logarithmic plot of mass residual $\lVert r_p^{n,i}\rVert_2^{rel}$ and flow residual $\lVert r_\mathbf{q}^{n,i}\rVert_2^{rel}$ decay for the mLRDD-scheme and monolithic Newton,  Picard, and L-scheme with wrong initial guess, with $\tau = 0.01$.}
        \label{fig:wrong guess}
    \end{figure}

    \paragraph{Computational Performance.} \label{par: performance MMS}
    Next, we investigate the computational performance of the mLRDD-scheme in comparison to the monolithic solvers. We monitor the average number of inner iterations, average computational time per inner iteration, and total runtime of \hyperref[example 1]{Example 1} and \hyperref[example 2]{Example 2}. To get insight into how simulation performance scales with the problem size, both examples are run for different mesh sizes. Also, as the mLRDD-scheme is well-suited for parallelization due to the independent subdomain problems within each inner iteration, each subdomain problem is solved in parallel using a separate worker process. This is achieved through a straightforward parallelization strategy, requiring minimal changes to the overall code structure. The performance of the parallelized approach is also compared against a serial implementation where the subdomain problems are solved sequentially. As seen in Table~\ref{tab:performance combined} for both examples, the mLRDD-scheme has comparable average time per iteration to Picard and the L-scheme at the coarsest mesh size, but outperforms the monolithic approaches for finer mesh sizes. The Newton method has the lowest average number of iterations, though it has the largest time per iteration. For total runtime, Newton and Picard perform the best due to their low iteration count, and the mLRDD-scheme generally outperforms the L-scheme due to its low average time per iteration. The parallelized mLRDD-scheme is generally faster than the serial version, though at the coarsest mesh, the overhead of parallelization can outweigh its benefit. We note that the current parallelization is not fully optimized, and further performance gains could be achieved through more advanced implementations. Overall, the mLRDD-scheme's advantage over the monolithic approaches becomes more pronounced as the problem size grows.
    
    \begin{table}[htbp]
        \centering
        \caption{Simulation performance metrics for the mLRDD-scheme and the monolithic Newton, Picard, and L-scheme, for \hyperref[example 1]{Example 1} and \hyperref[example 2]{Example 2}, at $T=1$ with $\tau=0.001$.}
        \vspace{0.5em}
        {\performanceCombined}
        \label{tab:performance combined}
    \end{table}

    \paragraph{Parameter study.} \label{par:parameter study}
    Finally, we performed a parameter study on the Robin-parameter $\alpha$, considering only \hyperref[example 1]{Example 1}. For this study, we set the maximum number of linear iterations to $i=100$. We first investigate the relationship between $\alpha$ and the stabilization parameter $L$ via a parameter sweep over different discretization levels. The results are presented in Figure~\ref{fig:alpha_vs_L}, where each heatmap shows the average number of iterations $i$ for different $(\alpha, L)$ pairs at a fixed discretization level $(\tau, h)$. The heatmaps show no spatial dependency for $\alpha$ and $L$, but a clear temporal dependency as the parameter ranges for convergence change with the timestep size $\tau$.  

    \begin{figure}
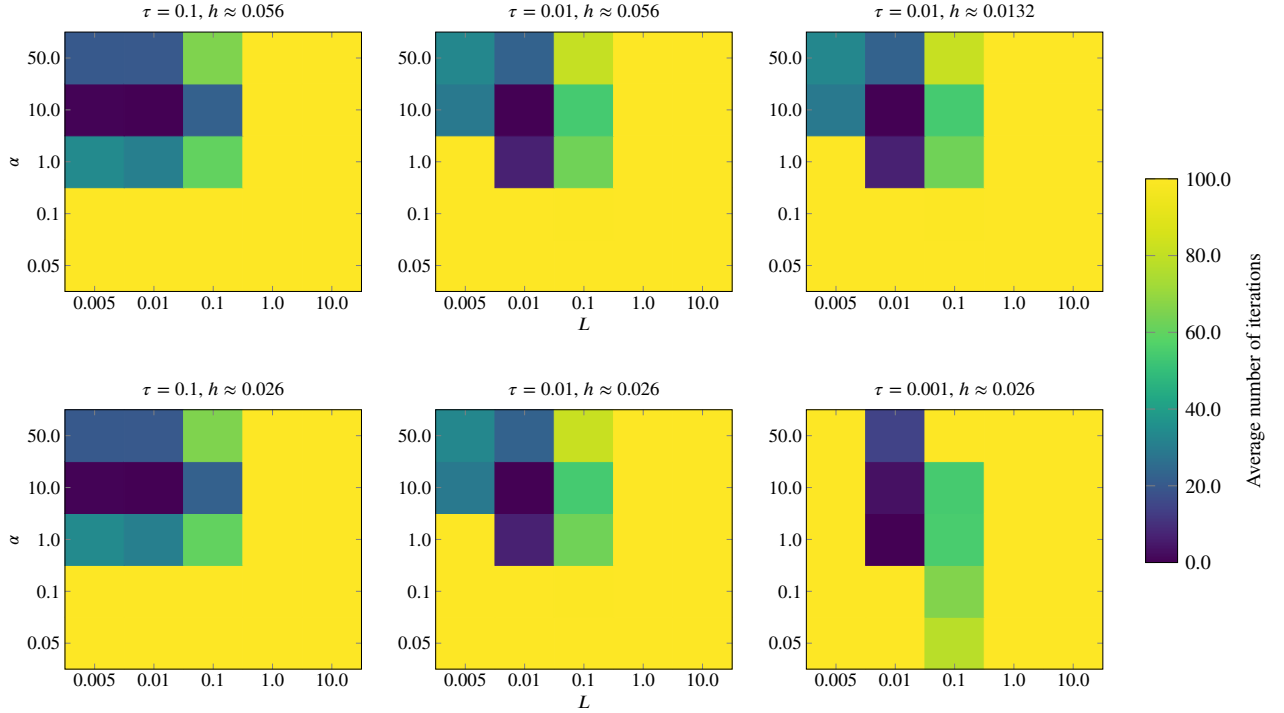

        \centering
        \begin{minipage}[c]{0.91\textwidth}
            \centering
            \begin{subfigure}[t]{0.32\linewidth}
                \centering
                \resizebox{1.00\linewidth}{!}{\usebox{\AlphaVsLOne}}
                \label{fig:alpha_vs_l_one}
            \end{subfigure}
            \begin{subfigure}[t]{0.32\linewidth}
                \centering
                \resizebox{1.00\linewidth}{!}{\usebox{\AlphaVsLTwo}}
                \label{fig:alpha_vs_l_two}
            \end{subfigure}
            \begin{subfigure}[t]{0.32\linewidth}
                \centering
                \resizebox{1.00\linewidth}{!}{\usebox{\AlphaVsLSix}}
                \label{fig:alpha_vs_l_three}
            \end{subfigure}
    
            \vspace{1ex}
    
            \begin{subfigure}[t]{0.32\linewidth}
                \centering
                \resizebox{1.00\linewidth}{!}{\usebox{\AlphaVsLFour}}
                \label{fig:alpha_vs_l_four}
            \end{subfigure}
            \begin{subfigure}[t]{0.32\linewidth}
                \centering
                \resizebox{1.00\linewidth}{!}{\usebox{\AlphaVsLFive}}
                \label{fig:alpha_vs_l_five}
            \end{subfigure}
            \begin{subfigure}[t]{0.32\linewidth}
                \centering
                \resizebox{1.00\linewidth}{!}{\usebox{\AlphaVsLThree}}
                \label{fig:alpha_vs_l_six}
            \end{subfigure}
        \end{minipage}%
        \begin{minipage}[c]{0.07\textwidth}
            \centering
            \resizebox{1.8\linewidth}{!}{\usebox{\AlphaVsLColorbar}}
        \end{minipage}
    
        \caption{Heatmaps of different Robin-parameter $\alpha$ and stabilization parameter $L$ pairs over different discretization levels for \hyperref[example 1]{Example 1}. Each heatmap displays average iterations $i$ per timestep for at a specific timestep size $\tau$ and mean cell diameter $h$ pair, at $T=0.5$.}
        \label{fig:alpha_vs_L}
    \end{figure}

    Next, we examine the relationship between $\alpha$ and the subdomain permeabilities by varying the absolute permeability $K_\ell$ in each subdomain $\Omega_\ell$. As the choice of absolute permeability in \hyperref[example 1]{Example 1} affects the Lipschitz constant of the saturation -- and thus the theoretical lower bound for the stabilization parameter $L$ to achieve convergence -- $L$ is chosen as $0.51 \max_{\ell\in\{1,2\}}L_{S_\ell}$. This choice will ensure that convergence is not hindered by an improperly scaled $L$, and gives consistency across parameter sweeps. In Figure~\ref{fig:alpha_vs_k}, each subplot corresponds to a fixed $K_1$ and each curve to a fixed $K_2$. For each $(K_1, K_2)$ pair, the curves exhibit a clear minimum at an intermediate value of $\alpha$, with iteration count increasing for both smaller and larger values. When $K_2$ is small, the iteration count hits the maximum threshold across all tested $\alpha$ values, indicating that no choice of Robin-parameter within the tested range satisfies extremely low permeability in $\Omega_2$. Figure~\ref{fig:optimal_alpha_heatmap} summarizes the main findings in Figure~\ref{fig:alpha_vs_k} by reporting the optimal $\alpha$, i.e., the value yielding the lowest average iteration count for each $(K_1, K_2)$ pair. The optimal Robin-parameter is primarily governed by the larger of the two permeabilities, where, as $\max(K_1, K_2)$ increases, the optimal $\alpha$ decreases systematically. For low permeabilities, this trend changes somewhat, hinting at a dependency on the permeability ratio $K_1 / K_2$ as well. This suggests that the contrast across the interface $\Gamma$ plays a role beyond the per-subdomain permeability, and $\alpha$ should be chosen per interface, in accordance with the permeability structure of the global domain, rather than treated as a fixed constant.
        
    \begin{figure}
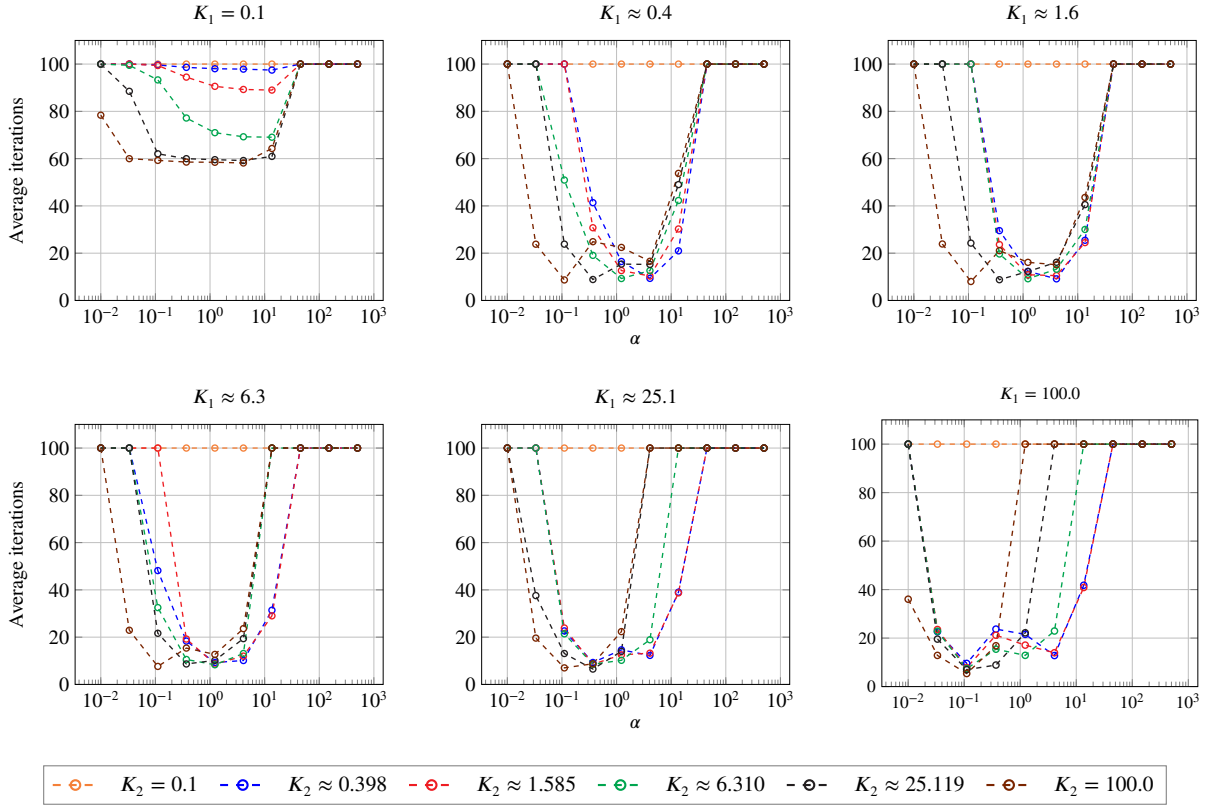

        \centering
        \begin{subfigure}[t]{0.32\textwidth}
            \centering
            \resizebox{1.00\linewidth}{!}{\usebox{\AlphaVsKOneA}}
            \label{fig:alpha_vs_k_one}
        \end{subfigure}
        \begin{subfigure}[t]{0.32\textwidth}
            \centering
            \resizebox{1.00\linewidth}{!}{\usebox{\AlphaVsKOneB}}
            \label{fig:alpha_vs_k_two}
        \end{subfigure}
        \begin{subfigure}[t]{0.32\textwidth}
            \centering
            \resizebox{1.00\linewidth}{!}{\usebox{\AlphaVsKOneC}}
            \label{fig:alpha_vs_k_three}
        \end{subfigure}

        \vspace{0ex}
                
        \begin{subfigure}[t]{0.32\textwidth}
            \centering
            \resizebox{1.00\linewidth}{!}{\usebox{\AlphaVsKOneD}}
            \label{fig:alpha_vs_k_four}
        \end{subfigure}
        \begin{subfigure}[t]{0.32\textwidth}
            \centering
            \resizebox{1.00\linewidth}{!}{\usebox{\AlphaVsKOneE}}
            \label{fig:alpha_vs_k_five}
        \end{subfigure}
        \begin{subfigure}[t]{0.32\textwidth}
            \centering
            \resizebox{1.00\linewidth}{!}{\usebox{\AlphaVsKOneF}}
            \label{fig:alpha_vs_k_six}
        \end{subfigure}
        
        \vspace{0ex}
        
        \centering
        \pgfplotslegendfromname{K2-legend}
        \caption{Effect of Robin-parameter $\alpha$ on number of iterations $i$ for different permeabilities, using \hyperref[example 1]{Example 1}: multiple semi-logarithmic plots of average number of iterations $i$ per timestep against different Robin-parameter $\alpha$ values. Each figure has a fixed absolute permeability $K_\ell$ for domain $\ell=1$, where each curve represent a specified $K_\ell$ for domain $\ell=2$, with $h\approx0.026$, $\tau=0.01$ at $T=0.5$}
        \label{fig:alpha_vs_k}
    \end{figure}

    \begin{figure}
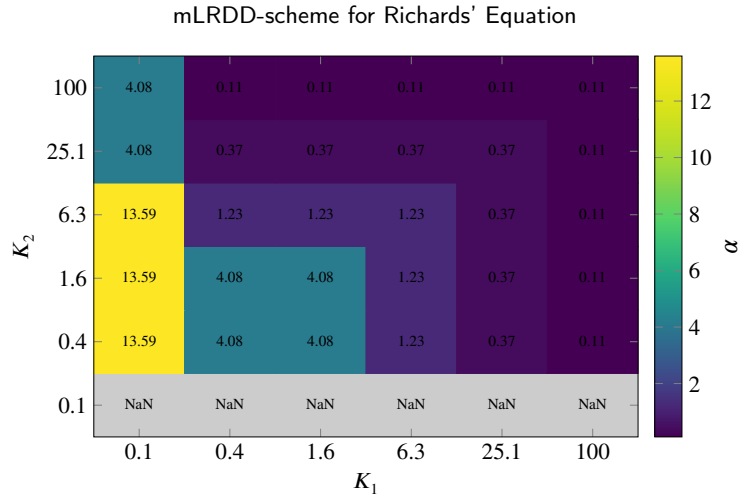

        \centering          
        \resizebox{0.60\linewidth}{!}{\usebox{\OptimalAlpha}}
        \caption{Optimal Robin-parameter $\alpha$ for different combinations of permeabilities, using \hyperref[example 1]{Example 1}: heatmap of different pairs of absolute permeability $K_\ell$ for domains $\ell=1$ and $\ell=2$. Each pair is colored by the optimal value of Robin-parameter $\alpha$, chosen from $\alpha$ giving the lowest average iteration $i$ for each $(K_1,K_2)$ pair in Figure~\ref{fig:alpha_vs_k}. NaN indicates that the maximum iteration count of $i=100$ was reached.}
        \label{fig:optimal_alpha_heatmap}
    \end{figure}
    
    \subsection{Example 3: Heterogeneous multi-domain 3D problem with van Genuchten parameterization} \label{example 3}
    To evaluate the performance of the proposed numerical scheme under realistic, physically demanding conditions, we adopt an extension of a recognized benchmark problem from \cite{haverkamp1977comparison, schneid2000hybrid} and \cite{list2016study}, amongst others. Unlike the manufactured solutions used to validate the scheme against theory, this test case evaluates the robustness of the solver against the severe nonlinearities inherent to multi-material soil physics.
    
    We consider a three-dimensional, highly heterogeneous porous medium, specifically
    \begin{equation*}
            \Omega = (-1, 1) \times (0, 1) \times (0, 1). 
    \end{equation*}
   To simulate complex stratigraphic layering and discontinuous material interfaces, the domain is partitioned into eight distinct subdomains, 
    $\{ \Omega_\ell \}_{\ell=1}^8$. Each subdomain is assumed to be internally homogeneous, with spatial coordinates defined as:
    \begin{align*}
        \Omega_1 &= (-1, 0) \times (0.5, 1) \times (0.5, 1), 
        & \Omega_5 &= (-1, 0) \times (0.5, 1) \times (0, 0.5),  \\
        \Omega_2 &= (-1, 0) \times (0, 0.5) \times (0.5, 1),
        & \Omega_6 &= (-1, 0) \times (0, 0.5) \times (0, 0.5), \\
        \Omega_3 &= (0, 1) \times (0.5, 1) \times (0.5, 1), 
        & \Omega_7 &= (0, 1) \times (0.5, 1) \times (0, 0.5), \\
        \Omega_4 &= (0, 1) \times (0, 0.5) \times (0.5, 1), 
        & \Omega_8 &= (0, 1) \times (0, 0.5) \times (0, 0.5).
    \end{align*}

    \begin{table}[htbp]
        \centering
        \caption{Subdomain materials, stabilization parameter $L$ for different subdomain types, and Robin-parameter $\alpha$ for different interface types, for \hyperref[example 3]{Example 3}.}
        \vspace{0.5em}
        {\exThreeDomains}
    \end{table}

    The domain consists of three different porous materials: Sandstone ($\Omega_{4,5,6}$), Silt Loam ($\Omega_{1,2,3}$), and Clay ($\Omega_{7,8}$). The varying flow capabilities of these materials create strong discontinuities at the interfaces, especially at intersecting interfaces. To govern the fluid flow, we utilize a van Genuchten--Mualem parameterization \cite{van1980closed} for the effective saturation $\Phi_\ell(p_\ell)$, the water saturation $S_\ell(p_\ell)$, and the relative permeability $k_\ell(S_\ell)$. The system is defined as
    \begin{equation*}
        \begin{aligned}
            \Phi_\ell(p_\ell) &= \begin{cases} 
                \dfrac{1}{\left(1 + (-\alpha_\ell p_\ell)^{\hat{n}_\ell}\right)^{m_\ell}}, & p_\ell < 0 \\ 
                1, & p_\ell \geq 0 
            \end{cases}, \quad m_\ell = 1 - \frac{1}{\hat{n}_\ell}, \\[5pt]
            S_\ell(p_\ell) &= S_{l,r} + (S_{l,s} - S_{l,r}) \; \Phi_\ell(p_\ell), \qquad
            k_\ell(S) = \sqrt{\Phi_\ell(p_\ell)} \left( 1 - \left( 1 - \Phi_\ell(p_\ell)^{\frac{1}{m_\ell}} \right)^{m_\ell} \right)^2,
        \end{aligned}
    \end{equation*}
    for $\ell \in \{1, \ldots, 8\}$. The material-specific parameters, alongside the non-dimensionalized collective terms utilized by the solver, can be seen in Table \ref{tab:vGM_parameters_expanded}.

    To drive the fluid flow, the system is subjected to a combination of Dirichlet and homogeneous Neumann boundary conditions. We define two Dirichlet boundary patches with prescribed pressures, thereby inducing flow through the resulting pressure gradients. The first patch, $\partial\Omega_{D_{1}}$, is located on the top surface of the domain, simulating localized infiltration. The second patch, $\partial\Omega_{D_{2}}$, is positioned on the lower-right boundary, acting as a drainage region when a certain internal pressure criterion is met. For the remainder of the exterior boundary, $\partial\Omega_{N}$, no-flow conditions are imposed, resulting in 
    \begin{equation*}
        \begin{aligned}
            \partial\Omega_{D_1} &= \{(x,y,z) \in \partial\Omega \mid z = 1,\; x < 0,\; y > 0.5 \}, &\qquad 
            \partial\Omega_{D_2} &= \{(x,y,z) \in \partial\Omega \mid x = 1,\; y < 0.5,\; z < 0.5 \}, \\[5pt]
            \partial\Omega_{D} &= \partial\Omega_{D_1} \cup \partial\Omega_{D_2}, &\qquad 
            \partial\Omega_{N} &= \partial\Omega \setminus \partial\Omega_{D}.
        \end{aligned}
    \end{equation*}
    Initially, the entire domain is prescribed a uniform, strongly negative pressure, representing a highly unsaturated (dry) state with zero initial flux. To simulate a wetting event, the boundary pressures are increased gradually via a linear ramp over a time period $\hat{t}$. After $t > \hat{t}$, the Dirichlet conditions are held constant. The complete initial and boundary value formulation is given by
    \begin{equation*}
        \begin{aligned}
            \tilde{p}(x,y,z,0) &= -1.0,
            &\qquad \tilde{\mathbf{q}}(x,y,z,0) &= (0,0,0)^T, \\[5pt]
            \tilde{p}(x, y, z, t) &= -1.0 + \min\!\left(\frac{t}{\hat{t}}, 1\right)
            \begin{cases}
                0.9 &\text{on } \partial\Omega_{D_1} \\
                1.0 - z &\text{on } \partial\Omega_{D_2}
            \end{cases},
            &\qquad \tilde{\mathbf{q}}(x, y, z, t) \cdot \mathbf{n} &= 0 \quad \text{on} \ \partial\Omega_N.
        \end{aligned}
    \end{equation*}

    \begin{figure}
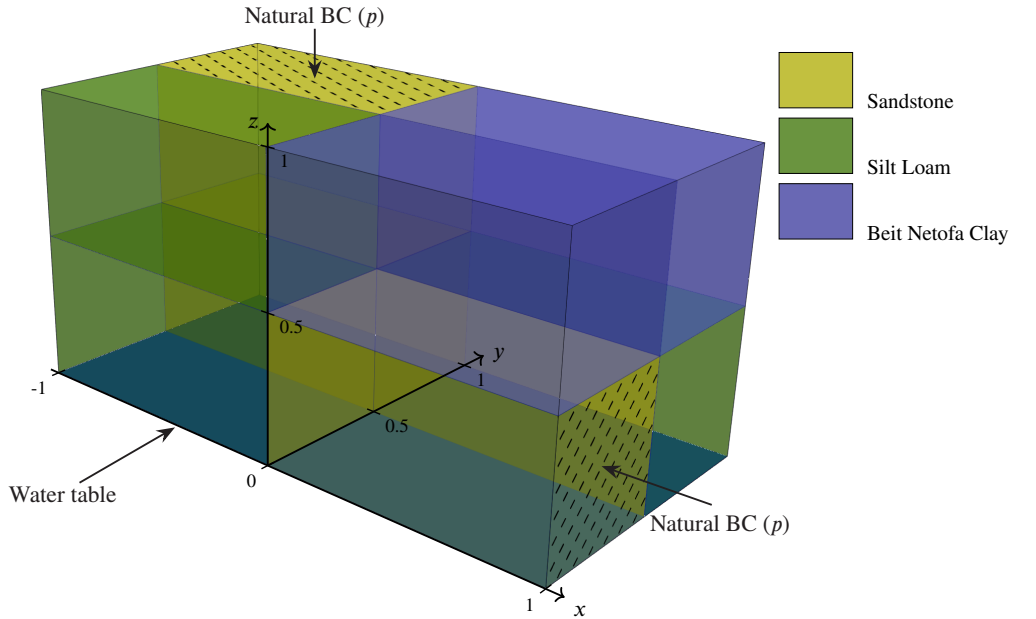

        \centering
        \resizebox{0.6\linewidth}{!}{\usebox{\exThree}}
        \\[5pt]
        \caption{Structure of the 3D domain with annotated boundary conditions, materials, and initial water table for \hyperref[example 3]{Example 3}.}
        \label{fig:layout example 3}
    \end{figure}
    
    \begin{table}[H]
        \centering
        \caption{The van-Genuchten--Mualem parameters for \hyperref[example 3]{Example 3}, including dimensionless collective terms based on scaling factors $p_c = 14.8 \times 10^3$ Pa, $L_c = 1.48$ m, and $t_c = 41440$ s.}
        \label{tab:vGM_parameters_expanded}
        \resizebox{\linewidth}{!}{\usebox{\vanGenuchtenParams}}
    \end{table}

    \subsubsection{Results}
    In the following paragraphs, we present the results from \hyperref[example 3]{Example 3}. The stabilization parameter $L$ is prescribed per subdomain type, i.e., each subdomain made up of the same material has the same $L$ value. Robin-parameter $\alpha$ is given per interface type; that is, each interface having the same adjacent materials shares a common $\alpha$ value.

    \paragraph{Water table.} \label{par:water table ex 3}
    Firstly, we observe the evolution of the water table ($p=0$) at different selected timesteps. As seen in Figure \ref{fig:watertable}, as the water enters the top and side sand patches (see Figure \ref{fig:layout example 3} for the domain structure), the water table rises first in the rearmost sand-column. Next, the foremost sand-block is filled, showing that the water enters the highly permeable sandstone first. Later, the water table ascends the silt loam regions, with minimal water penetration in the low permeable clay layer. In all, the fluid flow behavior is consistent with the initial and boundary conditions and soil material composition. 

    \begin{figure}
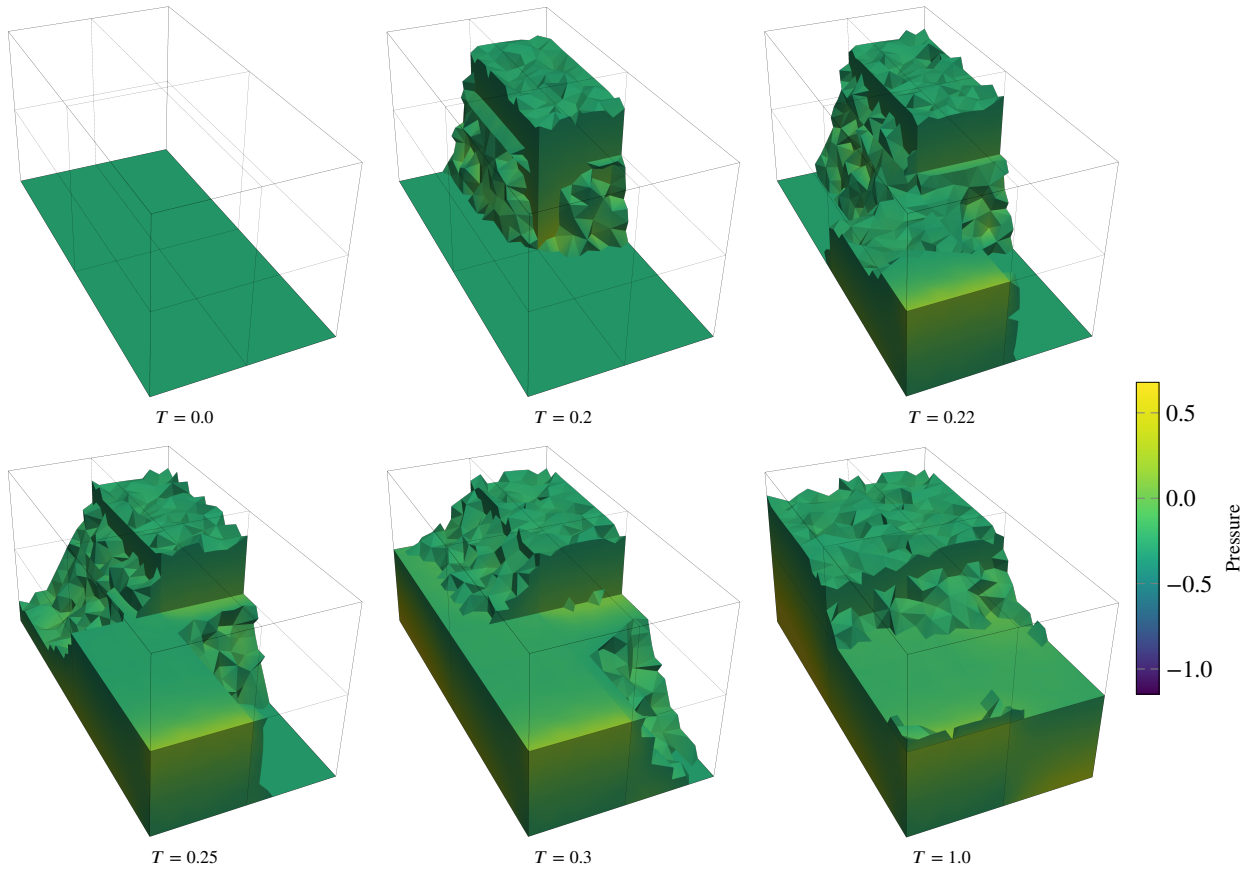

        \centering
        \resizebox{\textwidth}{!}{%
        \begin{tabular}{@{}ccc@{}l@{}}
            \usebox{\watertableZero} & \usebox{\watertablePointTwo}& \usebox{\watertablePointTwoTwo} & \multirow{2}{*}{\usebox{\watertableColorbar}} \\[2pt]
            $T=0.0$ & $T=0.2$ & $T=0.22$ &  \\[8pt]
            \usebox{\watertablePointTwoFive}& \usebox{\watertablePointThree} & \usebox{\watertableOne}      & \\[2pt] 
            $T=0.25$ & $T=0.3$ & $T=1.0$ & \\
        \end{tabular}
        }
        \caption{Water table ($p=0$) evolution for \hyperref[example 3]{Example 3} at selected times, with $h\approx0.122$, $\tau=0.01$.}
        \label{fig:watertable}
    \end{figure}

    \paragraph{Convergence of the iterative scheme.} \label{par:inner convergence ex 3}
    To evaluate the convergence of the iterative scheme for the highly demanding 3D problem, we use the same approach as for the 2D cases in Section \ref{subsection: MMS}. In Figure \ref{fig:inne convergence ex 3}, different error quantities are plotted for each iteration $i$ at a single timestep $n$. The errors are decreasing overall, showing close to linear convergence towards later iterations. In contrast to the simpler 2D problems, the flow residual is now the limiting factor for faster convergence.
    
    \begin{figure}
        \centering
        \resizebox{0.60\textwidth}{!}{\usebox{\iterErrorExThree}}
        \caption{Convergence of the iterative scheme for \hyperref[example 3]{Example 3}: semi-logarithmic plot of the mass residual $\lVert r_p^{n,i} \rVert_2^{rel}$ and flow residual $\lVert r_\mathbf{q}^{n,i}\rVert_2^{rel}$, the interface flux jump $\lVert \llbracket \mathbf{q}^{n,i} \cdot \mathbf{n} \rrbracket\rVert_{L^2(\Gamma)}$, and the Robin-variable iterate difference $\lVert g^{n,i} - g^{n,i-1}\rVert_{L^2(\Gamma)}^{rel}$, against iteration number $i$, at $T = 0.5$ with $\tau = 0.01$, $h \approx 0.122$.}
        \label{fig:inne convergence ex 3}
    \end{figure}

    The convergence of the mLRDD-scheme is also compared to the monolithic Picard and L-scheme by showing the mass and flow residual decay for each method at a single timestep in Figure \ref{fig:residuals ex 3}. The effect of mesh-refinement on the residual decay is also demonstrated by running the example on different mesh sizes. The Newton method is not shown in the results as it diverged for all tested discretization levels. The results show a clear mesh dependency on the residual for the monolithic methods; as the grid gets refined, the monolithic methods have extreme fluctuations in the residual for half of their iteration count, before reaching a steady decay. The mLRDD-scheme has the same mesh dependency for the flow residual; however, the mass residual has a more stable decay over different mesh sizes.

    \begin{figure}
      \centering
      \begin{subfigure}[t]{0.49\textwidth}
        \centering
        \resizebox{0.95\linewidth}{!}{\usebox{\residualPComparisonExThree}}
        \label{fig:p_residual}
      \end{subfigure}
      \hfill
      \begin{subfigure}[t]{0.49\textwidth}
        \centering
        \resizebox{0.95\linewidth}{!}{\usebox{\residualQComparisonExThree}}
        \label{fig:q_residual}
      \end{subfigure}
    
      \vspace{2ex}
      \centering
      \pgfplotslegendfromname{shared-residual-legend-ex-three}
    
      \vspace{1ex}
        \caption{Comparison of residual decay between methods for \hyperref[example 3]{Example 3}: Semi-logarithmic plot of mass residual $\lVert r_p^{n,i}\rVert_2^{rel}$ and flow residual $\lVert r_\mathbf{q}^{n,i}\rVert_2^{rel}$ decay for the mLRDD-scheme and monolithic Picard and L-scheme over different mesh sizes at $T=0.5$, with $\tau = 0.01$.}  \label{fig:residuals ex 3}
    \end{figure}

    \paragraph{Computational Performance.} \label{par:performance ex 3}
    Lastly, we look at the computational performance of the mLRDD-scheme in relation to the monolithic approaches, and the advantages of parallelization. As done for the 2D problems in Section \ref{subsection: MMS}, Table \ref{tab:performance ex 3} shows the average number of iterations $i$, the average time per iteration, and total runtime of different runs with progressively refined grids. As in the case of the previous results for the 3D problem, the Newton method is omitted, as it failed to converge. For the coarsest grids, Picard has the lowest total runtime due to its low iteration count. The mLRDD-scheme outperforms the monolithic schemes for the finest grid, with the parallelized approach having an approximately $31 \times$ and $37 \times$ faster per-iteration runtime than the L-scheme and Picard, respectively. This gives the parallelized mLRDD-scheme a total runtime of approximately $2.3$ hours, saving over $16$ and $35$ hours compared to Picard and the L-scheme, respectively. This significant reduction in computational time is most likely due to the increased resolution of sharp fronts and steep gradients for finer grids, especially at interfaces. This strengthens the global nonlinearity, presenting severe difficulties for the monolithic solvers. The mLRDD-scheme, however, avoids this problem by confining the nonlinearities to individual subdomains -- an advantage that becomes increasingly noticeable as the grid is refined. Figure~\ref{tab:speedup ex3} further illustrates the scaling behavior of the parallelized mLRDD-scheme, showing consistent speedup over the serial implementation as the number of worker processes increases. While the speedup does not scale linearly with increasing processes due to, e.g., process communication overhead, the parallelization strategy used requires minimal changes to the original implementation and yields a clear computational benefit. 

    \begin{table}[htbp]
        \centering
        \caption{Simulation performance metrics for \hyperref[example 3]{Example 3}, at $T=1$ with $\tau=0.01$.}
        {\performanceExThree} \label{tab:performance ex 3}
    \end{table}

    \begin{table}[htbp]
        \centering
        \caption{Speedup of the parallelized mLRDD-scheme based on the number of parallel processes for \hyperref[example 3]{Example 3}, at $T=1$ with $\tau=0.01$ and $h \approx 0.11$.}
        {\coresExThree} \label{tab:speedup ex3}
    \end{table}

    With this, we conclude the numerical results.

    \section{Conclusion}\label{sec:conclusion} 
    In this paper, we considered the mixed formulation of Richards' equation, modeling the flow of water through variably saturated porous media. For heterogeneous domains, the mixed form of Richards' equation yields a highly nonlinear, degenerate saddle-point problem that poses a triple challenge: a spatial discretization that must respect mass conservation and flux behavior, a linearization robust enough for sharp nonlinearities, and a strategy for managing strongly contrasting properties across material interfaces. Our proposed mixed L-scheme Robin-type domain decomposition scheme (mLRDD-scheme) addresses all three simultaneously. By discretizing in time with backward Euler and combining the L-scheme with non-overlapping Robin-type domain decomposition, the scheme isolates near-homogeneous linear subproblems, coupled through physically consistent interface conditions. The convergence of the scheme was proved under standard assumptions and mild constraints on the timestep size and stabilization/Robin-parameter. As the mLRDD-scheme itself is not tied to a particular choice of spatial discretization, any $H(\mathrm{div};\Omega)-L_2$ conforming method may be employed. Thus, by choosing $RT0-P0$ mixed finite elements as the spatial discretization, we obtained local mass conservation and normal flux continuity. 
    
    Numerical experiments in two and three spatial dimensions supported the theoretical results, demonstrating first-order convergence for all monitored quantities, and confirming continuity across the interfaces. Compared to standard monolithic approaches, the mLRDD-scheme was shown to be more robust than the Newton and modified Picard methods, particularly when handling strong material heterogeneities and poor initial guesses. Furthermore, it demonstrated shorter computational time than the monolithic solvers for highly nonlinear large-scale problems. As the subdomain problems are independent at each iteration, the scheme lends itself naturally to parallel solving. We implemented a simple parallelization strategy, demonstrating that the parallelized mLRDD-scheme yields an even more notable reduction in computational time. Finally, a parameter study revealed that the optimal stabilization parameter $L$ and Robin-parameter $\alpha$ are dependent on the timestep size but not on the mesh size. Moreover, $\alpha$ is shown to scale primarily with the maximum permeability and the permeability contrast between subdomains at the interface. 

    \section*{Software and data availability}
        The source code used to produce all numerical results presented in this paper 
        is openly available \href{https://doi.org/10.5281/zenodo.21995483}{here}. The source code is written in the programming language Python, with the source code requiring the software PorePy \cite{keilegavlen2021porepy} (see \url{https://github.com/pmgbergen/porepy}) and PyGeon \cite{pygeon} (see \url{https://github.com/compgeo-mox/pygeon}). For information related to the source code, contact \href{mailto:asmund.synnevag@uib.no}{asmund.synnevag@uib.no}. No experimental data is associated with this work; all results are generated directly by running the provided code. A Docker image is included in the repository to ensure full reproducibility of the computational environment.
        
    \section*{CRediT authorship contribution statement} 
        \textbf{Åsmund v.B. Synnevåg}: Conceptualization, Data curation, Formal analysis, Investigation, Methodology, Software, Validation, Visualization, Writing -- original draft, Writing -- review \& editing.
        \textbf{Wietse M. Boon}: Conceptualization, Formal analysis, Methodology, Resources, Software, Supervision, Validation, Writing -- original draft, Writing -- review \& editing.
        \textbf{Florin A. Radu}: Conceptualization, Methodology, Resources, Supervision, Validation, Writing -- original draft, Writing -- review \& editing.
        \textbf{Sarah E. Gasda}: Funding acquisition, Project administration, Writing -- review \& editing.
    
    \section*{Declaration of competing interests}
        The authors declare that they have no known competing financial interests or personal relationships that could have appeared to influence the work reported in this paper.
    
    \section*{Acknowledgments}
    This work was carried out as part of the MuPSI project (Multiscale Pressure-Stress Impacts on fault integrity for multi-site regional  CO$_2$ storage), funded through the Clean Energy Transition Partnership (CETP), project number CETP-2023-00298. Funding was provided by the Research Council of Norway (RCN), Scottish Enterprise, Dutch Research Council (NWO), Agencia Estatal de Investigaci\'{o}n (AEI), and the U.S. Department of Energy (DoE), with contributions from Storegga Ltd, Equinor ASA, Norske Shell AS, and EBN Capital BV.

    \appendix
    \section{Appendix} \label{sec:appendix}

     Let $\Omega \subset \mathbb{R}^d $, $d\in\{ 2,3 \}$ be an open, bounded domain with a Lipschitz-continuous boundary $\partial\Omega$. We assume that the boundary is partitioned into a Dirichlet part $\partial\Omega_D$ and a Neumann part $\partial\Omega_N$ such that $\partial\Omega = \overline{\partial\Omega}_D \cup \overline{\partial\Omega}_N$ and $\partial\Omega_D \cap \partial\Omega_N = \emptyset$. The domain $\Omega$ is partitioned into two non-overlapping Lipschitz subdomains $\Omega_1$ and $\Omega_2$ such that $\overline{\Omega} = \overline{\Omega}_1 \cup \overline{\Omega}_2$ and $\Omega_1 \cap \Omega_2 = \emptyset$. 
    Let $\partial\Omega_1$ and $\partial \Omega_2$ denote their respective boundaries, separated by the internal interface $\Gamma \coloneqq \partial \Omega_1 \cap \partial \Omega_2$, which forms a $(d-1)$-dimensional manifold in $\overline{\Omega}$. For simplicity, we assume that each subdomain borders (part of) the Dirichlet boundary, i.e., $\partial\Omega_\ell \cap \partial \Omega_D \ne \emptyset$ for each subdomain index $\ell \in \{1,2\}$. 
    Let $\mathbf{n}_\ell$ denote the outward unit normal vector on $\partial\Omega_\ell$; in particular, $\mathbf{n}_1 = -\mathbf{n}_2$ on $\Gamma$. 
    
    Let $L^2(X)$ denote the space of square-integrable functions on $X \in\{\Omega, \Omega_\ell, \Gamma\}$, equipped with the inner product $\langle u,v \rangle_X \coloneqq \int_X u \ v \ \mathrm{d}x$ (with $\mathbf u \cdot \mathbf v$ replacing $uv$ for vector-valued functions) and induced norm $\lVert u \rVert_{L^2(X)} \coloneqq \langle u,u \rangle_X^{1/2}$. Let $\langle u,v \rangle_{\Omega_\ell}$ denote the inner product on $L^2(\Omega_\ell)$, $\ell \in \{1,2\}$, and $\lVert u \rVert_{\Omega_\ell}$ the corresponding norm. We omit the subscript to refer to the global $L^2$ inner product and norm
    \begin{align} 
        \langle u, v \rangle &\coloneqq \sum_{\ell=1}^2 \langle u_\ell, v_\ell \rangle_{\Omega_\ell}, &
        \lVert u \rVert^2 &\coloneqq \sum_{\ell=1}^2 \lVert u_\ell \rVert_{L^2(\Omega_\ell)}^2. &
    \end{align}
    Similarly to inner products and norms, the omission of the subscript $\ell$ on any variable, e.g., $u$, denotes the global quantity defined over the entire domain $\Omega$. Next, we introduce the standard Sobolev spaces. The Hilbert space of scalar functions with square-integrable weak derivatives, and its subspace with vanishing trace on $\partial\Omega$, respectively, are defined as
    \begin{equation*}
        H^1(\Omega) \coloneqq \{ \phi \in L^2(\Omega) \mid \nabla \phi \in [L^2(\Omega)]^d \}, \qquad H_0^1(\Omega) \coloneqq \{ \phi \in H^1(\Omega) \mid \phi = 0 \text{ on } \partial\Omega \}.
    \end{equation*}
    For vector fields with square-integrable divergence, we define
    \begin{equation*}
        H(\mathrm{div}; \Omega) \coloneqq \{ \tilde{\mathbf{q}} \in [L^2(\Omega)]^d \mid \nabla \cdot \tilde{\mathbf{q}} \in L^2(\Omega) \},
    \end{equation*}
    equipped with the norm
    \begin{equation*}
        \lVert \tilde{\mathbf{q}} \rVert_{H(\mathrm{div}; \Omega)}^2 \coloneqq \lVert \tilde{\mathbf{q}} \rVert^2 + \lVert \nabla \cdot \tilde{\mathbf{q}} \rVert^2.
    \end{equation*}
On the interface $\Gamma$, we make use of the trace space $H_{00}^{1/2}(\Gamma)$ from \cite{lions2012non}, consisting of restrictions of $H_0^1(\Omega)$ functions to $\Gamma$.
    Its dual space is denoted by $H^{-1/2}(\Gamma)$, equipped with the norm
    \begin{equation*}
        \lVert \psi \rVert_{H^{-1/2}(\Gamma)} \coloneqq \sup_{\phi \in H^1_0(\Omega)} \frac{\langle \psi, \phi \rangle_{\Gamma}}{\lVert \phi \rVert_{H^1(\Omega)}}.
    \end{equation*}
    By $H_0^1(\Gamma)$, we denote the space of functions in $H^1(\Gamma)$ that vanish at the boundary of $\Gamma$. For the flux and pressure variables, we consider the following function spaces defined on the subdomains
    \begin{align} 
        Q_\ell  &\coloneqq \{
            \tilde{\mathbf{q}}_\ell \in H(\mathrm{div}; \Omega_\ell) 
            \mid (\tilde{\mathbf{q}}_\ell \cdot \mathbf{n}_\ell)|_{\Gamma} \in L^2(\Gamma), \ 
            (\tilde{\mathbf{q}}_\ell \cdot \mathbf{n}_\ell)|_{\partial \Omega_N \cap \partial\Omega_\ell} =0 \}, \\
        P_\ell  &\coloneqq  L^2(\Omega_\ell).
    \end{align}
    and we equip $Q_\ell$ with the graph norm
    \begin{equation*}
        \lVert\tilde{\mathbf{q}}_\ell\rVert^2_{Q_\ell}  \coloneqq  \lVert\tilde{\mathbf{q}}_\ell\rVert_{\Omega_\ell}^2 + \lVert\nabla \cdot \tilde{\mathbf{q}}_\ell\rVert_{\Omega_\ell}^2 + \lVert\tilde{\mathbf{q}}_\ell \cdot \mathbf{n}_\ell\rVert^2_{L^2(\Gamma)}.
    \end{equation*}
    The function space for the global flux variable on the full domain $\Omega$ is defined as: 
    \begin{equation*}
        \tilde{Q}  \coloneqq \{
            \tilde{\mathbf{q}} \in H(\mathrm{div}; \Omega) 
            \mid (\tilde{\mathbf{q}} \cdot \mathbf{n})|_{\partial \Omega_N} =0 \}.
    \end{equation*}
    We, moreover, define the global product spaces
    \begin{equation*} 
        Q \coloneqq  Q_1 \times Q_2, \qquad
        P \coloneqq  P_1 \times P_2. \qquad
    \end{equation*}
    Finally, we end this subsection with the function space for the interface variables, given by
    \begin{align}
        G &\coloneqq [L^2(\Gamma)]^2, &
        \langle g, \tilde g \rangle_G &\coloneqq \sum_{\ell=1}^2 \langle g_\ell, \tilde g_\ell \rangle_\Gamma, &
        \lVert g \rVert_G^2 &\coloneqq \sum_{\ell=1}^2 \lVert g_\ell \rVert_\Gamma^2.
    \end{align}
    Note that $G$ is the normal trace space of $Q$.

    \begin{table}
        \centering
        \caption{Notation and description for physical parameters, variables, and numerical parameters.}
        \usebox{\Notations}
        \label{tab:Notations}    
    \end{table}

    \bibliographystyle{cas-model2-names}
    
    \bibliography{references}

\end{document}